\documentclass{amsart}
\usepackage{verbatim,amsfonts,color,mathrsfs,stmaryrd,graphicx, mathdots}
\usepackage{xcolor}
\usepackage{tikz}
\usepackage{tikz-cd}
\usepackage{tkz-euclide}
\usetikzlibrary{patterns}   
\usetikzlibrary{positioning}
\usetikzlibrary{decorations.pathreplacing,calligraphy}

\theoremstyle{plain}
\newtheorem{TheirThm}{Theorem}[section]
\newtheorem{Thm}{Theorem}
\newtheorem{Cor}[Thm]{Corollary}

\newtheorem{Lem}[TheirThm]{Lemma}
\newtheorem{Prop}[TheirThm]{Proposition}

\theoremstyle{definition}
\newtheorem{Def}[TheirThm]{Definition}

\newtheorem{Eg}[TheirThm]{Example}

\theoremstyle{remark}

\usepackage{hyperref}

\begin{document}


\title{First return systems for some continued fraction maps}
\author{Thomas A. Schmidt}
\address{Oregon State University\\Corvallis, OR 97331}
\email{toms@math.orst.edu}
\date{9 November 2025}


\begin{abstract}  We prove a conjecture of Calta, Kraaikamp and the author: For all $n\ge 3$, each member of their one-parameter family of interval maps, denoted $T_{3,n,\alpha}$,  has its `first expansive return map'  of natural extension given by the first return map under the geodesic flow to a section of the unit tangent bundle  of the hyperbolic  surface uniformized by the underlying Fuchsian group $G_{3,n}$.    

To achieve the proof, we first prove the corresponding result for analogous  one-parameter families related to the Hecke triangle Fuchsian group $G_{2,n}$.   A direct comparison per $n$ of the $\alpha=1$ planar domains allows the Hecke group setting to provide sufficient information to prove the conjecture.

We also give details about the entropy functions for the Hecke triangle Fuchsian group maps,  $\alpha \mapsto h(T_{2,n,\alpha})$.   Each is continuous on $(0,1)$, increasing on $(0,1/2)$, decreasing on $(1/2,1)$, with a central interval of constancy.  We give precise formulas for the end points of the central intervals and also give precise formulas for the maximal entropy per family.   For fixed $\alpha$, the entropy of $T_{2,n,\alpha}$ goes to zero as $n$ tends to infinity.  
\end{abstract}

\maketitle

\tableofcontents

\section{Introduction}
\subsection{A conjecture proven, and results on entropy functions in the $(2,n)$ setting}    The Fuchsian triangle groups with one cusp and two elliptic conjugacy classes are, up to conjugation, of the form $G_{m,n}$ with  integers $n\ge m \ge 2$ (where if $m=2$ then $n>m$), where  $G_{m,n}$ is  the group generated by  

\begin{equation}\label{e:generators} 
A = \begin{pmatrix} 1& t\\
                                      0&1\end{pmatrix},\, B = \begin{pmatrix} \nu& 1\\
                                      -1&0\end{pmatrix},\,   C = \begin{pmatrix} -\mu& 1\\
                                      -1&0\end{pmatrix},
\end{equation}
for $\mu = \mu_m = 2 \cos \pi/m, \,\nu = \nu_n =  2 \cos \pi/n$ as well as  $t= \mu+\nu$.
 
For each of these groups and each $\alpha\in [0,1]$, \cite{CaltaKraaikampSchmidt} defined the  interval map $T_{m,n,\alpha}$ given in   \eqref{e:maps} below.   These maps are studied in  \cite{CaltaKraaikampSchmidt, CaltaKraaikampSchmidtContinEntrop, CaltaKraaikampSchmidtPfsErgodicity} with emphasis on the families  $T_{3,n,\alpha}$.  To an eventually expansive interval  map, \cite{CaltaKraaikampSchmidtContinEntrop} associate its {\em first pointwise expansive power map}, see \S~\ref{ss:firstExpPow} below.   They conjectured that the first pointwise expansive power map associated  to each  $T_{3,n,\alpha}$ is given by the first return map of the geodesic flow to a cross-section of the unit tangent bundle of the hyperbolic surface uniformized by $G_{3,n}$.    They proved the conjecture for $n=3$ and in general reduced the conjecture to showing that the volume of the unit tangent bundle  equals product of the Sinai-Kolmogorov measure theoretic entropy (hereafter simply: entropy) $h(T_{3,n,\alpha})$  with $\mu(\Omega_{3,n,\alpha})$, where $\Omega_{3,n,\alpha})$ is a planar domain for a 2-dimensional map naturally associated to $T_{3,n,\alpha}$ and $\mu$ is the measure defined in \eqref{e:muDefd} below.    We prove the conjecture by showing this equality.

\begin{Thm}\label{t:Main}   For all $n\ge 3$ and for all $\alpha \in (0,1)$, the product   
$h(T_{3,n, \alpha})\, \mu(\Omega_{3, n, \alpha})$ equals the volume of the unit tangent bundle of the hyperbolic orbifold uniformized by the triangle group $G_{3,n}$.    Furthermore,  the  first pointwise expansive power of  $T_{3, n, \alpha}$ has its natural extension given by
the first return of the geodesic flow to a cross section in the unit tangent bundle of the hyperbolic orbifold uniformized by $G_{3,n}$.
\end{Thm}
 Note that throughout the remainder of this paper, we will simply write ``surface" instead of  orbifold.

To reach these results, we prove similar results for  the natural analogs of our maps, denoted $T_{2,n,\alpha}$, but related to the Hecke triangle groups $G_{2,n}$.     Beyond what is strictly necessary to prove the conjecture, we also have the following. 
For $n \ge 3$ let  $R_n$  be the positive root of $x^2- (2-t_{2,n}) x - 1$ if $n$ is odd, and for even $n$ let $R_n = 1$.
\begin{Thm}\label{t:mIsTwoSimpleEntropyBehavior}   For each  $n\ge 3$  the entropy function $\alpha \mapsto h (T_{2,n,\alpha})$ is finite and  nonzero on $0<\alpha<1$.  Furthermore, this function is:  increasing on  $0< \alpha < 1- \frac{R_n}{t_{2,n}}$; is constant on $\big[1- \frac{R_n}{t_{2,n}}, 1+\frac{R_n}{t_{2,n}}\big)$;  and, decreases on  $1+\frac{R_n}{t_{2,n}}\le \alpha <1$.    Its maximal value is 
\[  \begin{cases}  \dfrac{\pi^2 \,(n - 2)}{2 n \ln \bigg(\dfrac{ 1 + \cos \pi/n}{  \sin  \pi/n}\bigg)}&\text{if}\;\; 2|n;\\
\\
                            \dfrac{\pi^2 \,(n - 2)}{2 n\ln \bigg(\dfrac{ 1 +R_n}{2 \sin \frac{\pi}{2 n}} \bigg)}&\text{otherwise}.
                                                 \end{cases}
\]

Finally, for any $0<\alpha<1$,   the limit as $n\to \infty$ of $h (T_{2,n,\alpha})$ equals zero.
\end{Thm}

\begin{figure}[h]
\scalebox{.4}{
\includegraphics{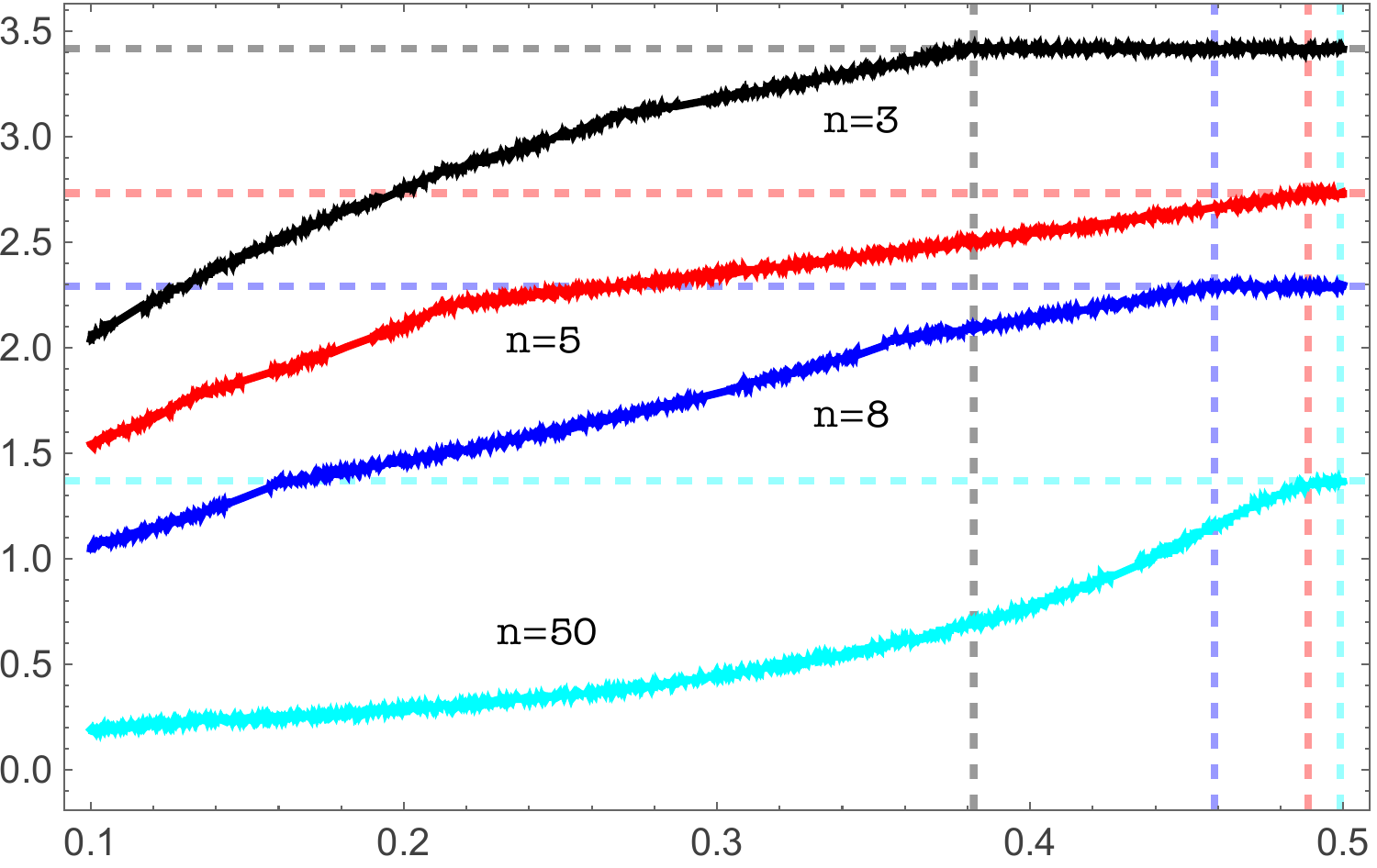}
}
\caption{Birkhoff sums approximation to entropy functions $\alpha \mapsto h(T_{2,n,\alpha})$, for $n\in \{3, 5,8,50\}$  and  $0.1< \alpha \le 1/2$.  Horizontal dashed lines at  maximum value as per Theorem~\ref{t:mIsTwoSimpleEntropyBehavior}; vertical dashed lines at $1-  R_n/t_{2,n}$ per $n$. }
\label{f:BirkEnt}
\end{figure}

\subsection{Motivation}    For much of the technical vocabulary mentioned here, see \S~\ref{s:back} Background.

 Although the connection between the geodesic flow on the unit tangent bundle of the modular surface and the regular continued fraction interval map goes back to at least Artin's \cite{Artin} in the 1920s, it was only in the 1980s and 1990s that one had proofs \cite{AdlerFlatto, Series, ArnouxCodage} that this interval map was a factor of the dynamical system defined by the first return under the geodesic flow to a section of the unit tangent bundle of the modular surface.

 In the ensuing years, various works have shown similar results for other interval maps also given piecewise by the fractional linear action of elements of a Fuchsian group of finite covolume.  Among these are \cite{ AdlerFlattoBack, AdlerFlattoBull, NakadaFord, MayerStroemberg, KatokUgarcoviciApplAbAndFlo, BocaMerriman,AbramsKatok, ArnouxSchmidtCross}.        In these works, consideration of the Sinai-Kolmogorov measure theoretic entropy (hereafter simply: entropy), $h(T)$,  of the interval map $T$ arises naturally.

The behavior of the entropy function for each of a parametrized family of  continued fraction-like interval maps was studied by Nakada \cite{N} in the setting of what are now known as the Nakada $\alpha$-continued fraction maps.   These maps, $T_{\alpha}$  with $\alpha \in [0,1]$,  are all piecewise by elements of the modular group $\text{PSL}_2(\mathbb Z)$ acting linearly fractionally.   
Nakada and Natsui \cite{NakadaNatsui} pointed to a phenomenon that now goes under the name of  `matching intervals', these are intervals  of $\alpha$ such that  $T^{k}_{\alpha}(\alpha) = T^{k'}_{\alpha}({\alpha-1})$ for pairs of positive integers $(k,k')$;   the function $\alpha \mapsto h(T_{\alpha})$ is  increasing, decreasing or constant on the interval 
according to whether   $k>k'$, $k<k'$, or $k=k'$.         Luzzi and Marmi  \cite{ LuzziMarmi}  strongly suggested that    $\alpha \mapsto h(T_{\alpha})$ is a continuous function of $\alpha$.    The continuity was proven a few years later  by both Carminati and Tiozzo \cite{CarminatiTiozzo} and  \cite{KraaikampSchmidtSteiner}.  Luzzi-Marmi also asked if every $T_{\alpha}$ is the factor of some cross section to the geodesic flow on the unit tangent bundle of the modular surface.    This was proven in \cite{ArnouxSchmidtCross}, using the techniques of the type sketched below in \S~\ref{ss:ArnouxMethod}. 
 
After Calta and the author \cite{CaltaSchmidt} considered various metric properties of the  interval maps $T_{3,n,0}\,$,     Calta, Kraaikamp and this author \cite{CaltaKraaikampSchmidt} then introduced the $T_{m,n, \alpha}$ with concentration on dynamics of the case of $m=3$.   By the way, the reason we avoided the case of $m=2$ was so as to not overlap with results already in the literature.   Continuity of the entropy functions  $\alpha \mapsto h(T_{3,n,\alpha})$ is shown for each $n\ge 3$ in \cite{CaltaKraaikampSchmidtContinEntrop}, using results of \cite{CaltaKraaikampSchmidtPfsErgodicity}.   

A  naive   hope  could be that any somewhat natural interval map $T$ which is given piecewise by the fractional linear action of a reasonable set of elements of a Fuchsian group of finite covolume should arise as the factor of the first return system of the geodesic flow on the unit tangent bundle of the hyperbolic surface uniformized by the group. 
However, already \cite{ArnouxSchmidtCross} gave an example of an interval map which, although given piecewise by the fractional linear action of elements of the modular,   is not   (as discoverable by the techniques of \S~\ref{ss:ArnouxMethod}) of  this type.       Still, the maps $T_{3,n,\alpha}$ are quite naturally, one could even say geometrically, defined in terms of elements of $G_{3,n}$. On the other hand,  these maps (in general) are only eventually expansive and the methods we use demand true expansiveness.   We were thus lead to introduce an acceleration of the maps.    The `first pointwise expansive power' of a map as defined in \cite{CaltaKraaikampSchmidtContinEntrop} is recalled in \S~\ref{ss:firstExpPow}  below.     With this in place, \cite{CaltaKraaikampSchmidtContinEntrop} made the conjecture that  each  $T_{3,n,\alpha}$ has its first pointwise expansive power of natural extension given by the first return by the geodesic flow to a cross-section  in the unit tangent bundle of the surface uniformized by $G_{3, n}$.

\subsection{Progress in \cite{CaltaKraaikampSchmidtContinEntrop} towards proof}
A proof in the case of $n=3$ and all values of $\alpha$ is given in \cite{CaltaKraaikampSchmidtContinEntrop}.  They also give numerical evidence for the conjecture for larger values of $n$, but that evidence is questionable  already for small $n$, see [\cite{CaltaKraaikampSchmidtContinEntrop}, \S~15.4].     We quote from that paper: ``The integrals which appear are naturally related to dilogarithmic functions and hence are notoriously difficult to evaluate exactly.    See  \cite{FriedSymDynTri} for  discussion of the appearance of the dilogarithm in a related setting." 

The conjecture (for general $n$ and $0<\alpha<1$) is reduced in \cite{CaltaKraaikampSchmidtContinEntrop}  in two steps.   Already to prove the continuity of the entropy functions, for each $n\ge 3$ and $0< \alpha< 1$ they give a planar extension $\mathcal T_{3,n,\alpha}: \Omega_{3,n,\alpha} \to \Omega_{3,n,\alpha}\,$, whose action on its first, thus $x$-, coordinate is given by   $T_{3,n,\alpha}$.    See \S~\ref{ss:planarBckgrd} for the basics of this type of planar extension.   The first reduction step shows that the conjectured result  holds for all $n\ge 3$ and all $0<\alpha<1$ for which  $h (T_{3,n,\alpha})  \mu(\Omega_{3,n,\alpha}) = \text{vol}(T^1(G_{3,n}\backslash \mathbb H))$, where this latter denotes the volume of the unit tangent bundle of the surface uniformized by $G_{3,n}$, see \eqref{e:volForm}.   This reduction step relies on:  Rohlin's entropy formula for entropy, see \eqref{e:rohlinEnt}; the fact that the invariant measure for each $T_{3,n,\alpha}$ is the marginal measure on $\Omega_{3,n,\alpha}$; and,  what we call `Arnoux's method', see \S~\ref{ss:ArnouxMethod}.  In brief, this last gives a measure preserving map sending (in the setting discussed) $\Omega_{3,n,\alpha}$ to $T^1(G_{3,n}\backslash \mathbb H)$ such that 
the integrand $\ln |T'(x)|$ in Rohlin's formula is also the length of the geodesic path following the geodesic flow from the image of $(x,y)\in \Omega_{3,n,\alpha}$ to the point corresponding to $\mathcal T_{3,n,\alpha}(x,y)$.    The second step  replaces showing the equality for  each $n$ and all $0<\alpha<1$ of $h (T_{3,n,\alpha})  \mu(\Omega_{3,n,\alpha})$ with the volume by showing the equality to this volume of the integral over (the infinite mass) $\Omega_{3,n,0}$ with respect to $\mu$ of the Rohlin integrand (for the map $T_{3, n, 0}$). We call such integrals {\em Rohlin integrals}.  That is, they show that the a priori varying values $h (T_{3,n,\alpha})  \mu(\Omega_{3,n,\alpha})$ are all equal to the Rohlin integral over $\Omega_{3,n,0}$.  See Lemma~\ref{l:integralsMatch} for an  analogous result when $m=3$ is replaced by $m=2$.  

The proof of the conjecture in the case of $m=n=3$  in  \cite{CaltaKraaikampSchmidtContinEntrop}    proceeds by partitioning $\Omega_{3,3,0}$ --- determined in \cite{CaltaSchmidt} --- into the union of two rectangles, with the value of the Rohlin integral being separately recognizable as equalling the volume of the unit tangent bundle of the modular surface, $\text{vol}(T^1(G_{2,3}\backslash \mathbb H))$.  Again from the well-known formula \eqref{e:volForm}, arithmetic then gives the result. 

 \subsection{Impetus for this paper}  On the occasion of a conference in April  2025 at the E.~Schr\"oder Institute in Vienna which was held in part to celebrate the retirement of C.~Kraaikamp, the author decided to revisit the conjecture of \cite{CaltaKraaikampSchmidtContinEntrop}.

\subsection{Approach in this paper}
Despite the aforementioned earlier hesitation of treating families of interval maps the Hecke group setting,  we prove the conjecture by first showing that the analogous conjecture when $m=2$ holds for any $n$ and $0<\alpha<1$.     
This path was taken because of an observation that each $\Omega_{3,n,1}$ is the union of a rectangle whose Rohlin integral equals  $\text{vol}(T^1(G_{2,3}\backslash \mathbb H))$ and its complement, with computation suggesting that this second piece would be given by a simple transformation applied to what was sure to be $\Omega_{2,n,1}$.     Furthermore,  several years ago Giulio Tiozzo encouraged the author to consider the various $T_{2,n,\alpha}$.  Recollection of that  interaction fully increased our belief that the work of \cite{CaltaKraaikampSchmidt, CaltaKraaikampSchmidtContinEntrop, CaltaKraaikampSchmidtPfsErgodicity} would all go through in the $m=2$ setting.     The key extra ingredient is a result of Nakada \cite{NakadaLenstra}, Theorem~\ref{t:HitoshiShowsGetHalf} below, which gives a relation of the entropy of each of the Rosen fraction maps to $\text{vol}(T^1(G_{2,3}\backslash \mathbb H))$.  As \cite{ArnouxSchmidtCommCF} used without proof, in the notation of this paper, this leads directly to $h (T_{2,n,1/2})  \mu(\Omega_{2,n,\alpha}) = \text{vol}(T^1(G_{2,n}\backslash \mathbb H))$.  We give a proof in \S~\ref{ss:symmRosen}.

 \subsection{Entropy maps in the $m=2$ setting} 
 We are purposely light throughout the paper in our treatment of the case $(m,n) = (2,3)$.    In fact,  the $T_{2,3,\alpha}$ are maps within the family of $G_{2,3} = \text{PSL}_2(\mathbb Z)$-related `$(a,b)$-continued fraction maps' introduced by Katok and Ugarcovici \cite{KatokUgarcoviciAbBegins}.    The $T_{2,3,\alpha}$ themselves, and in particular the entropy function $\alpha \mapsto h(T_{2,3, \alpha})$, were studied  in \cite{CarminatiIsolaTiozzo}.     Key to their approach is also a matching condition, which can be compared to be the matching condition of Lemma~\ref{l:shortRightId} (iv).   Their approach relies on `Farey words' in two letters to describe the matching intervals.   One can expect a direct formula to pass from  their two variable labeling of the matching intervals to our $(k,v)$ labels, but we do not pursue that here.     We do give Figure~\ref{f:changeOfMass} to allow the reader to check our planar domain against one figured in \cite{CarminatiIsolaTiozzo}, and an example to support our figure.       That all said,  it was rereading the work of \cite{CarminatiIsolaTiozzo} which suggested to us to take the detour to achieve Theorem~\ref{t:mIsTwoSimpleEntropyBehavior}.   
 
  We used Birkhoff sum approximations to find the approximate curves for the $\alpha \mapsto h(T_{2,n, \alpha})$ shown in Figure~\ref{f:BirkEnt}.   We first saw such ideas in \cite{Choe}, and in the setting of continued fraction-like interval maps in \cite{LuzziMarmi}.  See   \cite{Tiozzo} for results of convergence of Birkhoff sum approximations in a very related setting.    In fact, the Birkhoff sum approximations led to discovering the error in \cite{BKS}  which we point out and correct  in \S~\ref{ss:Correction}.  Note that Figure~\ref{f:BirkEnt} suggests that there is a change in concavity of the functions as $n$ goes from small to large values; we have not pursued that matter here.

\subsection{Outline}    Section~\ref{s:back} is a rather long background section, where we give elementary definitions and reminders through the less elementary matters of Arnoux's method and the setting of \cite{CaltaKraaikampSchmidt, CaltaKraaikampSchmidtContinEntrop, CaltaKraaikampSchmidtPfsErgodicity}, as well as state a theorem of Nakada \cite{NakadaLenstra}.    Section~\ref{s:ResultsMisTwo} presents the results in the setting of $m=2$; the section begins with a statement of results and then a somewhat detailed outline of its subsections. Section~\ref{s:the Shift} gives the shift from $\Omega_{2,n,1}$ to $\Omega_{3,n,1}$ and its implications.  Section~\ref{s:theConjectureHolds!} gives the final lines of the proof of Theorem~\ref{t:Main}, thus showing that the conjecture from \cite{CaltaKraaikampSchmidtContinEntrop} holds.

\section{ Background}\label{s:back}

 We give background, including repeating background from \cite{CaltaKraaikampSchmidt, CaltaKraaikampSchmidtContinEntrop, CaltaKraaikampSchmidtPfsErgodicity}.    
 
\subsection{The groups, the intervals, and the functions} 
   We consider  the following triangle Fuchsian groups.   Fix integers $n\ge m \ge 2$ (where if $m=2$ then $n>m$), and let $\mu = \mu_m = 2 \cos \pi/m, \,\nu = \nu_n =  2 \cos \pi/n$ as well as  $t= \mu+\nu$.  Thus, 

\[ t  := t_{m,n} = 2 \cos \pi/m + 2 \cos \pi/n.\]
  
We denote by  $G_{m,n}$ the group generated by

\begin{equation}\label{e:generators} 
A = \begin{pmatrix} 1& t\\
                                      0&1\end{pmatrix},\, B = \begin{pmatrix} \nu& 1\\
                                      -1&0\end{pmatrix},\,   C = \begin{pmatrix} -\mu& 1\\
                                      -1&0\end{pmatrix}\,,
\end{equation}

\medskip
\noindent
and note that $C= AB$.    We work projectively,  hence $B, C$ are of order $n,m$ respectively while $A$ is of infinite order.    That is, $G_{m,n}$ is a Fuchsian triangle group of signature $(0; m,n,\infty)$.

 \bigskip
 Fix  $\alpha \in [0,1]$ and define 
\begin{equation}\label{e:intervals} 
 \mathbb I_{m,n, \alpha}  := \mathbb I_{\alpha} = [\,(\alpha - 1)t, \alpha t\,)\,.
 \end{equation}
\bigskip 

\noindent
Let 
\begin{equation}\label{e:maps} 
  T_{\alpha} = T_{m,n,\alpha}: x \mapsto A^{k} C^{l}\cdot x, 
\end{equation}
\bigskip 

\noindent
 where as usual, any  $2 \times 2$ matrix {\small$\begin{pmatrix}a & b \\ c & d\end{pmatrix}$} acts on real numbers by {\small$\begin{pmatrix}a & b \\ c & d\end{pmatrix} \cdot x = \dfrac{a x + b}{c x + d}$}, and  
 
\begin{itemize} 
\item  $l\in \mathbb N$ is minimal such that $C^{l} \cdot x \notin \mathbb I$;

\medskip 
\item  \[k = -\lfloor (C^{l} \cdot x)/t + 1 - \alpha\, \rfloor\,.\]

\end{itemize}

\medskip 
\noindent
We consider $T_{\alpha}$ as a map on the closed interval taking values in the half-open interval $ \mathbb I_{\alpha}$,
\[  T_{\alpha}:  [\,(\alpha - 1)t, \alpha t\,] \to \mathbb I_{\alpha}\,.\]
When $\alpha = 0$ and $(m,n) = (3, n)$ this gives the (unaccelerated) maps treated in \cite{CaltaSchmidt},  the maps with  $(m,n) = (3, n)$ and $0<\alpha<1$ are studied in some detail in  \cite{CaltaKraaikampSchmidt, CaltaKraaikampSchmidtPfsErgodicity}.    That the $T_{m,n,\alpha}$ lead to continued fraction-like expansions of real numbers is shown in [\cite{CaltaKraaikampSchmidt}, \S~1.5].

 \subsection{Two fundamental formulas for entropy}
The Kolmogorov--Sinai  measure theoretic entropy, which as stated above we refer to simply as entropy  and usually denote in the form $h(T)$ ---  is an invariant of a dynamic system, which roughly speaking measures its complexity.   Rohlin introduced the notion of the natural extension system to aid in the study of entropy  and  showed that the original system and its natural extension share entropy values.   

For a dynamical system $(X, T, \mathscr B, \mu)$ of finite measure and a subset $E \subset X$ of positive measure, the {\em induced transformation} on $E$ is  $T_E: E \to E$ given by $T_E(x) = T^k(x)$ where $k \in \mathbb N$ is minimal such that $T^k(x) \in E$.  (By the Poincar\'e Recurrence Theorem, the set of $x \in E$ such that there is some such $k$ has full measure in $E$, and in fact one defines $T_E$ to be the identity on the complement of this subset.)  We set $\mu_E$ to be the restriction to $E$ of $\mu$ scaled by $1/\mu(E)$.   This allows one to define a dynamical system for $T_E$.   The following is a key tool in the study of planar extensions. {\em Abramov's Formula}   states that the entropy of the induced system on $E$ is the quotient of the entropy of the original system divided by the measure of $E$, in short
\begin{equation}\label{e:AbramForm}
 h(T_E)= h(T)/\mu(E).
\end{equation}

One also has {\em Rohlin's Entropy Formula} for an interval map $T: \mathbb I \to \mathbb I$ which is ergodic with respect to an invariant probability measure $\nu$  and such that the derivative $T'$ exists $\nu$-almost everywhere,  (see say \cite{DenkerKellerUrbanski}):
\begin{equation}\label{e:rohlinEnt}
h(T) = \int_{\mathbb I} \ln \vert\, T'(x)\,\vert\, d\nu.
\end{equation}

\subsection{Piecewise M\"obius maps, cylinders, planar extensions}\label{ss:planarBckgrd}  
For  
\begin{equation}\label{e:justM}
M = \begin{pmatrix} a&b\\c&d\end{pmatrix}
\end{equation}
 in $\text{SL}(2, \mathbb R)$    and $z \in \mathbb C$, one has $M\cdot z = \frac{az+b}{cz+d}$.   This fractional linear action allows us to identify $M$ with a M\"obius function.   We use mainly the restriction of this action to the real numbers.   Note that $M$ and $-M$ act identically, thus we may consider that  
the group $\text{PSL}(2, \mathbb R)$,  the quotient group of $\text{SL}(2, \mathbb R)$ by  $\{\pm \text{I}\}$,  is acting.

We consider various functions $T$ on a subinterval $\mathbb I \subset \mathbb R$ with values in $\mathbb I$ such that  
there is a partition $\mathbb I= \cup_{\beta \in \mathcal B}\, K_\beta$ with  $T(x) = M_{\beta}\cdot x$ for all $x \in K_\beta$, where each $M_{\beta}$ is in $\text{SL}_2(\mathbb R)$ and acts fractional linearly.  (In general, one can also allow $M$ of negative determinants, but this is unnecessary in this paper.)   We will assume that each $K_{\beta}$ is an interval and is taken as large as possible.   We call these $K_\beta$ the (rank one) {\em cylinders} for $T$.   Similarly, a cylinder of rank $m>1$ is the largest interval on which $T^m$ is given by the action of some $M_{\beta_m}\cdots M_{\beta_1}$.    We say that the  cylinder $K_{\beta}$ is {\em full} if $T(K_{\beta}) = \mathbb I$.

The standard number theoretic planar map associated to   $M \in \text{SL}_2(\mathbb R)$ is  
\[
\mathcal{T}_M(x,y) :=  \bigg( M\cdot x,  RMR^{-1}\cdot y\,\bigg)\, \quad  \text{for}\;\; x \in \mathbb R\setminus \{M^{-1}\cdot \infty\},\ y \in \mathbb R\setminus \{(RMR^{-1})^{-1}\cdot \infty\}\,,
\]
where 
\begin{equation}\label{e:justR}
R = \begin{pmatrix}0&-1\\1&0\end{pmatrix}.
\end{equation}
 
An elementary Jacobian matrix calculation verifies that the measure  $\mu$ on $\mathbb R^2$ given by
\begin{equation}\label{e:muDefd}
d\mu = \dfrac{dx\, dy}{(1 + xy)^2}
\end{equation}
is (locally) $\mathcal{T}_M$-invariant.  

 For  a piecewise M\"obius interval map  $T$ we then set  
\begin{equation}\label{e:2DlocMap}
\mathcal{T}(x, y) = \mathcal{T}_{M_{\beta}}(x,y)    \quad  \text{whenever}\;\; x \in K_\beta,\ y \in \mathbb R\setminus \{N^{-1}\cdot \infty\}\,.
\end{equation}
Suppose that  $\Omega \subset \mathbb R^2$ projects onto the interval $\mathbb I$ and is a domain of bijectivity  of $\mathcal T$,  that is $\mathcal T$ is bijective on $\Omega$ up to $\mu$-measure zero.  Let $\mathcal B$ be the Borel algebra of $\Omega$, we then call the system $(\mathcal T, \Omega, \mathcal B, \mu)$  a {\em planar extension} for $T$.   We will occasionally abuse this terminology and say that $\Omega$ or $\mathcal T$ is the planar extension.  Similarly, we say  $T$ has a {\em positive planar extension}  when $T$ has a planar extension with $0<\mu(\Omega)<\infty$.

We will have occasional need to use a planar extension of $T$ for which the invariant measure is Lebesgue measure.   Let 
\begin{equation} \label{eqZmap}
\mathcal Z(x,y) = (x, y/(1+xy))
\end{equation}
and  for each $M= \begin{pmatrix} a&b\\c&d\end{pmatrix}$ as above, let  $\widehat{\mathcal T}_M = \mathcal Z \circ \mathcal T_M\circ \mathcal Z^{-1}$.  Then 
\begin{equation} \label{eqLebMap}
\widehat{\mathcal T}_M(x,y) = ( M\cdot x, \,  (c x+d)^2 y - c (cx + d) \,).
\end{equation}

\phantom{here}\\  
\noindent
{\bf Convention}   Throughout, we will allow ourselves the minor abuse of using adjectives such as injective, surjective and bijective to mean in each case {\em up to measure zero}, and thus similarly where we speak of disjointness and the like we again will assume the meaning being taken to include the proviso ``up to measure zero" whenever reasonable.\\  
 
\subsection{Hyperbolic upper half-plane, geodesic flow, Fuchsian groups}\label{ss:ModGpSurfGeoFlo}
The upper half-plane is $\mathbb H= \{z=x+iy \,\vert\, x,y \in \mathbb R, y>0\}$, equipped with the hyperbolic metric on $\mathbb H$, whose element of arclength squares to be $ds^2 = (dx^2 + dy^2)/y^2$.   The group $\text{SL}(2, \mathbb R)$ acts by fractional linear maps, which are isometries of 
$\mathbb H$.      The geodesics of $\mathbb H$ are either vertical lines and semi-circles, whose naive extensions meet the boundary real line perpendicularly.

The action of $\text{SL}_2(\mathbb R)$ on $\mathbb H$ extends to an action on the unit tangent bundle $T^1 \mathbb H$.   This  allows one to identify $T^1 \mathbb H$ with $\text{PSL}_2(\mathbb R)$.    The {\em geodesic flow} on the unit tangent bundle is an action of the real numbers: given a nonnegative real number $t$ and a unit tangent vector $u$, the unit tangent vector   $t\cdot u$ is obtained by following the unique geodesic passing through $u$ in the positive direction for arclength $t$ and taking the unit tangent vector to this oriented geodesic at the new basepoint.  If $t<0$, we follow the geodesic in the opposite direction.    In terms of $\text{PSL}_2(\mathbb R)$, the geodesic flow for time $t$  is given by  sending $A\in \text{PSL}_2(\mathbb R)$ to $A g_t$,   where $g_t = \begin{pmatrix} e^{t/2}&0\\0&e^{-t/2}\end{pmatrix}$.

  A {\em Fuchsian group} $\Gamma$ is a discrete (with respect to the natural topology) subgroup of $\text{PSL}(2, \mathbb R)$ and in particular, it is a subgroup acting properly discontinuously on $\mathbb H$.   The quotient $\Gamma\backslash \mathbb H$ has a unit tangent bundle, given by equivalence classes (thus $\Gamma$-orbits) of unit tangent vectors of $\mathbb H$.  The geodesic flow descends so that for $t\in \mathbb R$,    $[A] \in \Gamma\backslash \text{PSL}_2(\mathbb R)$ is sent to $[A g_t]$, where the square brackets here denote $\Gamma$-cosets represented by the given elements.

\subsection{Volume of unit tangent bundles}  The volume of the unit tangent bundle of a surface uniformized by a Fuchsian group $\Gamma$ equals $\pi$ times  the hyperbolic area of a fundamental domain for $\Gamma$.  Each of the $G_{m,n}$ with $2 \le m$ and $3 \le n$   is a hyperbolic triangle group.  The fundamental domain is then the union  of a hyperbolic triangle of angles $0, \pi/m/ \pi/n$ with the reflection over one of its edges.   The area of the fundamental domain is twice that of the triangle, but the area of a hyperbolic triangle equals $\pi$ minus the sum of its angles.  Therefore,  

\begin{equation} \label{e:volForm}
\text{vol}(T^1(G_{m,n}\backslash \mathbb H)) = 2 \pi \; (\pi - \pi/m - \pi/n) = 2 \pi^2 \dfrac{mn - m- n}{mn}\,.
\end{equation}   
In particular,   $\text{vol}(T^1(G_{2,n}\backslash \mathbb H)) =   \pi^2 \,(n - 2)/n$ and $\text{vol}(T^1(G_{3,n}\backslash \mathbb H)) = 2 \pi^2 (2n - 3)/(3n)$.

\subsection{Arnoux's method for cross sections to a geodesic flow}\label{ss:ArnouxMethod}

\subsubsection{Cross sections to a measurable flow,  Arnoux's transversal}\label{sss:crossFlowArnouxTrans}        
Let $(X, \mathscr B, \mu)$ be a measure space and $\Phi_t$ a measure preserving flow on $X$,  that is $\Phi:  X \times \mathbb R \to X$ is a measurable function such that for $\Phi_t(x) = \Phi(x,t)$, $\Phi_{s+t} = \Phi_s \circ \Phi_t$.   Then $\Sigma \subset X$ is a  {\em measurable cross section} for  the flow $\Phi_t$ if:     (1)  the flow orbit of almost every point meets $\Sigma$; (2) for almost every $x\in X$ the set of times $t$ such that $\Phi_t(x) \in \Sigma$ is a discrete subset of $\mathbb R$; (3)   for every Borel subset (in the subspace topology)   $A \subset \Sigma$ and for every  $\tau> 0$, {\em flow box}  $A_{[0,\tau]}\ :=  \{ \Phi_t(A)\, \mid \, A \in \Sigma,  t \in [0, \tau]\,\}$ is $\mu$-measurable.  

The {\em return-time function} $r= r_\Sigma$ is $r(x) =\inf\{t>0:\Phi_t(x)\in\Sigma\}$ and the {\em return map} $R:\Sigma\to\Sigma$ is defined by $R(x)= \Phi_{r(x)}(x)$. 
The {\em induced measure} $\mu_\Sigma$ on $\Sigma$ is defined from flow boxes: one sets 
$\mu_\Sigma(A)= \frac{1}{\tau} \mu(A_{[0,\tau]})$ for any  $0<\tau<\inf_{x\in A}\{r(x)\}$.

\medskip

\noindent
{\bf Convention}:  In all that follows,  we write {\em cross section} to denote measurable cross-section.
 \smallskip

A flow $\Phi_t$ is {\em ergodic} if for any invariant set either it or its complement is of measure zero.    
  The fundamental result of Hopf states that if a Fuchsian group $\Gamma$ is of finite covolume, then the geodesic flow is ergodic with respect to the natural measure on the unit tangent bundle of $\Gamma\backslash \mathbb H$; see his reprisal in \cite{Hopf}.    A flow is {\em recurrent} if the $\Phi$-orbit of almost every point meets any positive measure set infinitely often.  By the Poincar\'e Recurrence Theorem, an ergodic flow on a finite measure space is recurrent.    Given a cross-section, its first return-time transformation  $(\Sigma, \mathscr B_\Sigma, \mu_\Sigma, R_\Sigma)$ is ergodic  whenever the flow is.  
  
 There is a natural measure on the unit tangent bundle $T^1 \mathbb H$:  the  {\em Liouville measure} is given  as the product of the hyperbolic area measure on $\mathbb H$ with the length measure on the circle of unit vectors at any point.   Liouville measure is (left- and right-) $\text{SL}_2(\mathbb R)$-invariant,  and  thus gives Haar measure on $G$.   In particular, this measure is invariant for the geodesic flow.     With standard normalizations, Liouville measure agrees with the Riemannian volume form.  These both descend modulo any Fuchsian group $\Gamma$.

  Arnoux, see say \cite{ArnouxSchmidtCross},  found an elementary manner to map Lebesgue planar extensions into the unit tangent bundle of appropriate surfaces so as to find cross sections to the geodesic flow.   This often allows one to express   an original interval map's system as a factor.  For $(x,y) \in \mathbb R^2$, let 
\[ A(x,y) =   \begin{pmatrix} x&xy-1\\1&y\end{pmatrix}.\]  
  {\em Arnoux's transversal}, $\mathscr A$, is the projection to $\text{PSL}_2(\mathbb R)$ of the set of all  $A(x,y)$. 
  The complement to   $\{[A g_t] \,\vert\, A\in \mathscr A, t \in \mathbb R\}$  is a null set for  Liouville measure.  Furthermore Liouville measure restricts to   $\mathscr A$ to be a constant multiple of $dx\,dy\, dt$.

We repeat a definition and observation from \cite{ArnouxSchmidtCommCF}. 
\begin{Def}\label{def:Tau}  For $M \in \text{SL}(2, \mathbb R)$ and $x \in \mathbb R\,$ such that $M \cdot x \neq \infty\,$, let 
\[\tau(M,x) := -2 \,\log |c x + d\,|\,\,\]
where $(c,d)$ is the bottom row of  $M\,$ as usual.   
\end{Def}

Elementary calculation shows that $\tau$ induces on $\text{PSL}(2, \mathbb R)$ a {\em cocycle}, in the following sense:
\[\tau(MN, x) = \tau(M, Nx) + \tau(N, x)\]
whenever all terms are defined and we choose each projective representative such that the corresponding $cx + d$ is positive.

  For ease of legibility, set $t_0 =  \tau(M,x)$.  Then  matrix multiplication shows
\begin{equation}\label{e:arnouxFlow}  M A  g_{t_0} =  A(\,\widehat{\mathcal T}_M(x,y)\,).\end{equation}

  Given a Fuchsian group $\Gamma$, we have the function
\[
\begin{aligned}
 \mathscr P =  \mathscr P_{\Gamma}: \mathbb R^2 &\to \Gamma\backslash\text{PSL}(2, \mathbb R)  \\
               (x,y) &\mapsto  [A(x,y)],
\end{aligned}               
 \]
 where again square brackets denote cosets.   Note that this is measure preserving when we use Lebesgue measure on $\mathbb R^2$ and the measure given by $dx\,dy$ inherited by  the projection $ \mathscr A_{\Gamma}$, of $\mathscr A$ to $\Gamma\backslash\text{PSL}(2, \mathbb R)$.   When we use $\mathscr P$ we will often commit the abuse of writing $A$ to represent the corresponding coset.

\subsubsection{Arnoux's method}    
  Now suppose that $T$ is a piecewise (determinant one) M\"obius interval map.  
 We say that the {\em group associated to $T$} is the group $\Gamma_T$ generated by the M\"obius transformations of $T$; thus,    $\Gamma_T = \langle M_\beta, \beta \in \mathcal B  \rangle \subset \text{PSL}_2(\mathbb R)$.      
 
 If $T$ has a positive planar extension, we conjugate via $\mathcal Z$  of \eqref{eqZmap} to the Lebesgue planar extension of $T$.  We then apply the measure preserving $\mathscr P$, with $\Gamma = \Gamma_T$, to $\mathcal Z (\Omega)$.  When $\Gamma_T$ is a Fuchsian group, this is a subset of $T^1(\Gamma_T\backslash \mathbb H)$;  using the fact that any Fuchsian group has countably many elements,  in \cite{ArnouxSchmidtCommCF} it is shown that $\mathscr P$ is injective up to measure zero on $\mathcal Z (\Omega)$. 
 
 Arnoux's method is illustrated by the following result.   When  $T$ is a piecewise determinant one M\"obius interval map and $x \in \mathbb I$, we let $\tau(x) = \tau_T(x) = \tau(M, x)$ where $T(x) = M\cdot x$.    The result here combines (\cite{ArnouxSchmidtCommCF} Theorem~5.4, Corollary~1, and Proposition~4). 
\begin{TheirThm}\label{t:detOneSetting} [Arnoux's Method] Let $T$ be a piecewise determinant one M\"obius interval map with positive planar extension,    and suppose that $\Gamma_T$ is Fuchsian group  of finite covolume.   Then 
\[ \Sigma = \mathscr P_{\Gamma}(\,\mathcal Z (\Omega)\,)\]
is a cross section to the geodesic flow on $T^1(\Gamma_T\backslash \mathbb H)$.  Furthermore,  
the system defined by 
\[\begin{aligned}
\phi: \Sigma \;\;\;&\to \;\Sigma\\
  {}      [A(x,y)] &\mapsto [M A(x,y) g_{\tau(x)}],
 \end{aligned}
 \]
 with $M$ such that $T(x) = M\cdot x$, 
 is an extension of $T: \mathbb I \to \mathbb I$.    Moreover, $\phi$ agrees with the first return map of the geodesic flow to $\Sigma$  if and only if $T$ is ergodic,   expansive,  and with entropy satisfying $h(T) \mu(\Omega_f) =  \emph{vol}(\,   T^1( \Gamma_T\backslash \mathbb H)\, )$.  When this holds, the   first return to $\Sigma$ gives a natural extension to $T$. 
\end{TheirThm}
The initial statement of the theorem is shown by considering  \eqref{e:arnouxFlow} with the projection.    The subset $\{[A(x,y) g_t] \,\vert\, A(x,y) \in \Sigma, 0\le t\le \tau(x)\} \subset T^1(\Gamma_T\backslash \mathbb H)$ is invariant under the geodesic flow;  Hopf's result implies that this is all of $T^1(\Gamma_T\backslash \mathbb H)$ up to measure zero.       If $\phi$ agrees with the first return map,  then the map to $\mathscr A$ is such that the flow is expansive in the $x$-direction and contracting in the $y$-direction ---  recall that the map of \eqref{eqLebMap} preserves Lebesque measure  ---  and one can argue as in \cite{ArnouxSchmidtCommCF} that the first return system is indeed the natural extension.   The ergodicity of the flow implies that $\phi$ is ergodic, it then follows that $\widehat{\mathcal T}, \mathcal T$ and hence $T$ itself are ergodic.   The veracity of the equation involving the entropy is also in  \cite{ArnouxSchmidtCommCF}; in brief,  $\tau(x)$ is simultaneously  the arclength of the geodesic path following the flow line from $ [A(x,y)]$ to its image under $\phi$ and the (piecewise form of the) integrand in Rohlin's formula \eqref{e:rohlinEnt}. Recall that $\nu$ is the marginal measure of $\mu$ on $\Omega$, and that both $\mathcal Z$ and $\mathcal P_{\Gamma}$ are measure preserving.  That in this setting one has a natural extension can be shown, for instance, using the main result of \cite{ArnouxSchmidtNatExt}.

\subsection{Nakada's entropy formula for   Rosen fractions}   Recall that, for each $n \ge 3$,   $\nu_n = 2 \cos \pi/n$;  set $\mathbb I_n =   [-\nu_n/2, \nu_n/2)$.  
Letting, as in \cite{NakadaFord}, $\lfloor  x \rfloor_n = a  \nu_n $ such that $x -a \nu_n \in  \mathbb I_n$, Rosen's \cite{Rosen} continued fraction  map on $\mathbb I_n$ is given by $f(x) = f_n(x)= |1/x| - \lfloor |1/x| \rfloor_n$ (with $x = 0$  fixed).    Rosen introduced these maps to study the Hecke triangle Fuchsian groups, the $G_{2,n}$ in our notation. 

Using a combination of hyperbolic geometrical considerations, Diophantine approximation, and ergodic theory,   Nakada \cite{NakadaLenstra} obtained the following.    We slightly reformulate.
\begin{TheirThm}[Nakada 2010]\label{t:HitoshiShowsGetHalf}  Let  $f_n$ denote the Rosen continued fraction map of integral index $n \ge 3$ and $\Omega_n$ the planar domain as given in \cite{BKS}.   Then  the product of the entropy of $f_n$ times   $\mu(\Omega_n)$  equals $1/2$ times the volume of the unit tangent bundle of the surface uniformized by the Hecke triangle Fuchsian group of index $n$, $G_{2,n}$. 
\end{TheirThm} 

\subsection{Setting of \cite{CaltaKraaikampSchmidt}}\label{ss:CKS}      
In \cite{CaltaKraaikampSchmidt} one has a detailed study of the interval maps  $T_{3, n, \alpha}$.  We tersely summarize their notation and results.

 The parameter interval is naturally partitioned, 
 $[0, 1] = \{0\} \cup (0, \gamma_n)\cup [\gamma_n,  )\cup \{1\}$,  where $\alpha < \gamma_n$ if  $\forall x \in [\ell_0(\alpha), r_0(\alpha)]$ one has $T_{n,\alpha}(x) = A^k C\cdot x$ for some nonzero integer $k$. 
 We call the set of $\alpha< \gamma_n$  the {\em small $\alpha$}.  
For small $\alpha$, the only possibly non-full cylinders are the leftmost cylinder (whose left endpoint is $\ell_0(\alpha)\,$) and the rightmost (with right endpoint $r_0(\alpha)\,$).  
 
     For $x \in [\ell_0(\alpha), r_0(\alpha)]$ we call  $d_{\alpha}(x)= k$ the (simplified) $\alpha$-digit of  $x$, and say that $x$ lies in the cylinder $\Delta_{\alpha}(k)$.   To  each $x$ we then associate the sequence of its $\alpha$-digits.   The $T_{\alpha}$-orbits of the endpoints $\ell_0(\alpha), r_0(\alpha)$ are of extreme importance in our discussions,  one writes $\underline{b}_{[1, \infty)}^{\alpha}$ for the $\alpha$-digits  of $\ell_0(\alpha)$ and  replaces the lower bar by an upper bar,  $\overline{b}_{[1, \infty)}^{\alpha}$ for $r_0(\alpha)$.   We also label entries in these orbits by $\ell_0, \ell_1, \dots$ and $r_0, r_1, \dots$.   
   
\subsubsection{Expansions given words, synchronization intervals, tree of words}\label{ss:ExpanSynTree} 
For small $\alpha$,  we define parameter intervals on which an initial portion of the expansion of the $r_0(\alpha)$ are fixed.   That is, there is a common prefix of the  $\overline{d}_{[1, \infty)}^{\alpha}$.     Given a word $v = c_1 d_1 \cdots d_{s-1}c_s$ of positive integers, let  
 $\overline{S}(v) = \sum_{i = 1}^s\, c_i + \sum_{j=1}^{s-1} d_j$ 
and for each  $k \in \mathbb N$  let 
\begin{equation}\label{e:overDkV} 
\overline{d}(k,v) = k^{c_1}, (k+1)^{d_1},\cdots,  (k+1)^{d_{s-1}},k ^{c_s}, \;\;\; \text{when}\;\; v = c_1 d_1 \cdots d_{s-1}c_s.
\end{equation}   The corresponding subinterval of parameters is  
\begin{equation}\label{e:defIkV} 
I_{k,v} = \{ \alpha\mid \overline{d}_{[1, \overline{S}(v) ]}^{\alpha} = \overline{d}(k,v)\,\}\,;
\end{equation}
 in words, this is the set of $\alpha$ for which  the $\alpha$-expansion (in simplified digits)  of 
$r_0(\alpha)$ has $\overline{d}(k,v)$ as a prefix. 
Thus, setting   
\begin{equation}\label{e:rKv}
R_{k,v}  =    (A^kC)^{ c_s}\; (A^{k+1}C)^{d_{s-1}}(A^kC)^{c_{s-1}}\cdots (A^{k+1}C)^{d_1} (A^kC)^{c_1},
\end{equation}    for $\alpha \in I_{k,v}$ one has  $T_{\alpha}^{\overline{S}(v)}(\, r_0(\alpha) \,)= R_{k,v} \cdot  r_0(\alpha)$. The left endpoint of $I_{k,v}$ is denoted $\zeta_{k,v}$, one finds that $R_{k,v} \cdot  r_0( \zeta_{k,v}) =  \ell_0( \zeta_{k,v})$.   We define 
\begin{equation}\label{e:defJayKv} 
J_{k,v} = [\zeta_{k,v}, \eta_{k,v}),
\end{equation}
 where  with $L_{k,v} = C^{-1}ACR_{k,v}$ we have $L_{k,v}\cdot r_0(\eta_{k,v}) = r_0(\eta_{k,v})$.   Confer (\cite{CaltaKraaikampSchmidt}, Figure~4.2).

A main result of \cite{CaltaKraaikampSchmidt} is that each $J_{k,v}$ is what is now usually called a {\em matching interval}\,:   for all $\alpha$ in the interior of $J_{k,v}$, the $T_{\alpha}$-orbits of $\ell_0(\alpha)$ and of $r_0(\alpha)$ meet and do so in a common fashion.   We furthermore showed that the complement in $(0, \gamma_n)$ of the union of the $J_{k,v}$ is of measure zero.   Key to this was determining a maximal common prefix of the $\underline{d}_{[1, \infty)}^{\alpha}$ for $\alpha \in J_{k,v}$.   For that, we used 
$w = w_{n} = (-1)^{n-2}, -2,  (-1)^{n-3},-2$ and for $k$ fixed, we let  $\mathcal C = \mathcal C_k =(-1)^{n-3}, -2, w^{k-1}$ and $\mathcal D = C_{k+1}$, and defined 
$\underline{d}(k,v) = w^k, \mathcal C^{c_1-1} \mathcal D^{d_1}\cdots \mathcal D^{d_{s-1}}\mathcal C^{c_s}, (-1)^{n-2}$.    This is the common prefix, and also 
\begin{equation}\label{e:lowerDetaExpansion}
 \underline{d}{}^{\eta_{k,v}}_{[1,\infty)}  =  \overline{ w^k, \mathcal C^{c_1-1} \mathcal D^{d_1}\cdots \mathcal D^{d_{s-1}}\mathcal C^{c_s}, (-1)^{n-3},-2}\,,
 \end{equation} 
where the overline indicates a period.   We denote the length as a word in $\{-1, -2\}$ of $\underline{d}(k,v)$  by $\underline{S}(k,v)$.  There is an expression for $L_{k,v}$ related to $\underline{d}(k,v)$ in a manner similar to how \eqref{e:rKv} relates $R_{k,v}$ to  $\overline{d}(k,v)$.   The aforementioned synchronization is of the form
\begin{equation}\label{e:synchronizationExplicit}
T_{\alpha}^{\overline{S}(v)+1}(\, r_0(\alpha) \,)= T_{\alpha}^{\underline{S}(v)+1}(\, \ell_0(\alpha) \,).
 \end{equation} 

The words $v$ that we use form a tree, $\mathcal V$.   
For this, we first define $v'$ such that any $\overline{d}_{[1, \infty)}^{\eta_{k,v}}$ equals  $\overline{d}(k,v(v')^{\infty}\,)$, where the latter is a slight abuse of notation.    The following is [\cite{CaltaKraaikampSchmidt}, Definition~4.8].

\begin{Def}\label{d:defVprime}   
For each $s > 1$ and each   word $v = c_1d_1\cdots c_{s-1}d_{s-1} c_s$, define 

\[ v' =  \begin{cases} 1 (c_1 - 1)d_1c_2 \cdots c_{s-1}d_{s-1} c_s &\text{if}\;\;\; c_1\neq 1\,,\\
                               (d_1+1) c_2 \cdots c_{s-1}d_{s-1} c_s &\text{otherwise}\,.
                              \end{cases}
 \]
We interpret this also to mean that when $v = c$ with $c>1$ then   $v' = 1 (c -1)$, and when $v=1$ then $v' = 1$.
\end{Def}

  The following is [\cite{CaltaKraaikampSchmidt}, Definition~4.10].
\begin{Def}\label{d:thetaQ}   Set $\Theta_{-1}(c_1) = c_1+1$ and $\Theta_q(1) = 1q1 $ for $q \ge 1$.  For  $c>1$,  set  $\Theta_q(c) = c [ 1 (c-1)]^q 1c$ for any $q \ge 0$.  (To avoid double labeling and also to stay within our desired language,   $\Theta_0(1)$ is undefined;   note that $\Theta_1(1) = 111$, compare with $\Theta_0(c) = c 1 c$ for $c>1$.)

We now recursively define values of the operators  $\Theta_q$.  Suppose $v = \Theta_p(u) = u v''$ for some $p\ge 0$ and some suffix $v''$.   Then   define for any $q\ge 0$ 
\[ \Theta_q(v) = v (v')^q v'' \,.\]
\end{Def} 

\begin{Def}\label{d:descendents}   Let $\mathcal V$ denote the set of all words obtainable from $v=1$ by finite sequences of applications of the various $\Theta_q$. 
\end{Def}    

There is a type of self-similarity of $\mathcal V$ which allows the explicit definition of the {\em derived words} operator $\mathscr D$ such that (again for  general $v$)  $\mathscr D\circ \Theta_q (v) = \Theta_q\circ\mathscr D(v)$.    In general $\mathscr D$ decreases the length of words while preserving various properties, and hence assists in induction proofs.   For example, we applied it to prove  that every $v \in \mathcal V$ is a palindrome.   For more on the tree, see (\cite{CaltaKraaikampSchmidt},  \S 4.2).

 \subsection{First expansive power maps}\label{ss:firstExpPow}  We recall the following directly from \cite{CaltaKraaikampSchmidtContinEntrop}.
 
\begin{Def}\label{d:firstExpansiveReturn}  
Let $T$ be a function on  an interval  $I$, with  an invariant probability measure $\nu$ which is equivalent to Lebesque measure. Suppose that $T$ is eventually expansive, thus there is some $r \in \mathbb N$ such that $T^r$ is expansive.    By the Well Ordering Principle of the integers, for each $x\in I$ there is  a least $\ell(x) \in \mathbb N$ such that $T^{\ell(x)}(x)$ is expansive in the sense that the derivative here is greater than 1 in absolute value.    We define the {\em first pointwise expansive power of $T$}  as the map  $U: I\to I$  sending each $x$  to $T^{\ell(x)}(x)$.
\end{Def}
 
Of course, $\ell(x) \le r$ for all $x \in I$ and also $U(x) = T(x)$ for all $x \in E$, where $E$ is the maximal subset of $I$ on which $T$ itself is expansive.     For each $k \le r$, let $E_k \subset I$ be the set on which $\ell(x) = k$.    Of course,  the $E_k$ give a finite partition of $I$.   We have that $|(T^k)'(x)|>1$ for all $x \in E_k$;  
by  the Chain Rule, $|(T^k)'(x)| = \prod_{i=0}^{k-1}\, |T'(\, T^i(x)\,)|$.   The {\em minimality} of $\ell(x)$ hence implies both  $T(E_k)\subseteq \cup_{i=1}^{k-1} \, E_i$ for $k\ge 2$ and $T^{k-1}(E_k) \subset E_1$ for all $k$.

Note that in the setting where $I$ is an interval and $T$ is given piecewise by M\"obius transformations,  then also  $U$ is so given.

 \section{Results for the setting of $m=2$}\label{s:ResultsMisTwo}
Besides the aforementioned Theorem~\ref{t:mIsTwoSimpleEntropyBehavior}, we show the following results specific to the $m=2$ setting.  
 
\begin{Thm}\label{t:mIsTwoEntropyCont}  For each $n\ge 3$, the function assigning to $\alpha \in (0,1)$ the measure theoretic entropy of $T_{2,n,\alpha}$ is continuous and is invariant under  $\alpha \mapsto 1 -\alpha$.  
\end{Thm}

For each  $n\ge 3$ and $\alpha \in [0,1]$ let  $\Omega_{2,n,\alpha}$ be as defined in Subsections \S~\ref{ss:PlanarExtMatch},  \ref{ss:conMuOm}  and \ref{ss:omZeroOne}, according as $\alpha$: lies in the interior of a matching interval; is any other value in $(0,1)$; is in $\{0,1\}$.   

\begin{Thm}\label{t:mIsTwoSymm}  For each $n\ge 3$ and $\alpha \in [0,1]$,  the map $(x,y) \mapsto (-x, -y)$ sends $\Omega_{2,n,\alpha}$ to $\Omega_{2,n,1-\alpha}$.  
\end{Thm}

\begin{Thm}\label{t:mIsTwoGotVol}   For all  $n\ge 3$ and $\alpha \in [0,1]$ we have
\[ \int_{\Omega_{2,n,\alpha}}\,  \log \vert T'_{2,n,\alpha}(x) \vert\, d \mu  = \emph{vol}(T^1(G_{2,n}\backslash \mathbb H)).\] 

For  $\alpha \in (0,1)$,  the measure theoretic entropy of $T_{2,n,\alpha}$ satisfies 
\[h (T_{2,n,\alpha})  \mu(\Omega_{2,n,\alpha}) = \emph{vol}(T^1(G_{2,n}\backslash \mathbb H)).\]
\end{Thm}

  In this section, we first consider in \S~\ref{ss:Symmetry} the basic symmetry $\alpha \mapsto 1-\alpha$ between interval  maps and also between their associated planar maps.  This symmetry naturally points to $\alpha = 1/2$ as a special case; in \S~\ref{ss:symmRosen} we prove the  second statement of Theorem~\ref{t:mIsTwoGotVol} in the case of $\alpha = 1/2$, a result used without proof in  \cite{ArnouxSchmidtCommCF} and which is indeed virtually a direct result of Nakada's Theorem \ref{t:HitoshiShowsGetHalf}. For each $n\ge 3$, the exact value of the $\mu$-mass of $\Omega_{2,n,1/2}$ is twice that of the planar extension region given by \cite{BKS}.  Due to the result of \S~\ref{ss:symmRosen}, Birkhoff approximations for the entropy of certain $T_{2,n, 1/2}$ made us suspicious of their  $\mu$-mass value in the odd $n$ case.  In \S~\ref{ss:Correction} we correct their formula.  
  
   The remainder of the section, other than \S~\ref{ss:incrDecr},  shows that the techniques and results of \cite{CaltaKraaikampSchmidt, CaltaKraaikampSchmidtContinEntrop, CaltaKraaikampSchmidtPfsErgodicity}  hold in the $n=2$ setting.   In \S~\ref{ss:incrDecr} we give the main results underpinning Theorem~\ref{t:mIsTwoSimpleEntropyBehavior}.  That is, we determine the behavior of  $\alpha \mapsto \Omega_{2,n, \alpha}$ along matching intervals, determine the middle matching interval along which the function is constant, as well as the maximal value.  
   
   In \S~\ref{ss:conMuOm} we complete the proof of all but the first statement of Theorem~\ref{t:mIsTwoSimpleEntropyBehavior} by showing the continuity of  $\alpha \mapsto \mu(\Omega_{2,n \alpha})$ on all of $(0,1)$. 
   In \S~\ref{ss:omZeroOne}, we determine $\Omega_{2,n,0}$ and $\Omega_{2,n,1}$ for all $n\ge 3$. Combined with the earlier constructions of the various $\Omega_{2,n,\alpha}$, this completes the proof of Theorem~\ref{t:mIsTwoSymm} .    In an initial part of \S~\ref{ss:conEntConRoh}, we use the explicit two dimensional regions $\Omega_{2,n,\alpha}$ and `word processing' to allow a use of Abramov's Formula, showing continuity of $\alpha \mapsto h(T_{2,n, \alpha})$ on all of $(0,1)$; this completes both the proof of   Theorem~\ref{t:mIsTwoSimpleEntropyBehavior}, and due to the symmetry,  also that of Theorem~\ref{t:mIsTwoEntropyCont}.   In the later part of \S~\ref{ss:conEntConRoh}, we consider Rohlin integrals to complete the proof of Theorem~\ref{t:mIsTwoGotVol}.

\subsection{Symmetry under $\alpha \mapsto 1- \alpha$}\label{ss:Symmetry}
The basic observation that for  $x \notin \mathbb Z$ one has $\lfloor -x \rfloor = -1 -\lfloor -x \rfloor$ leads to the following.

\subsubsection{Interval maps}
If $T_{2,n,\alpha}(x) \neq \ell_0(\alpha)$, one finds $T_{2,n,1-\alpha}(-x) = - T_{2,n,\alpha}(x)$.  Furthermore,   if $T_{2,n,\alpha}(x) = \ell_0(\alpha) = (\alpha-1)t$ then $T_{2,n,1-\alpha}(-x) =  \ell_0(1-\alpha)  = - \alpha t$.

For more detail, when $T_{2,n,\alpha}(x) \neq \ell_0(\alpha)$, one has $\frac{-1}{ t x} + 1 - \alpha  \notin \mathbb Z$ and hence
\[
\begin{aligned} 
- T_{2,n, \alpha}(x) &= \frac{1}{x} + \lfloor  \frac{-1}{ t x} + 1 - \alpha \rfloor\, t &=:  \frac{1}{x} + k t\\
                               &=  \frac{1}{x} + ( 1 + \lfloor  \frac{-1}{ t x}   - \alpha \rfloor)\, t\\
                               &= \frac{1}{x} + ( 1 + - 1 - \lfloor  \frac{1}{ t x}   + \alpha \rfloor)\, t\\
                               &= \frac{1}{x}   - \lfloor  \frac{1}{ t x}   + \alpha \rfloor \, t\\
                               &= \frac{1}{x}   - \lfloor  \frac{1}{ t x}   + 1 - (1- \alpha) \rfloor \, t\ &=:   \frac{1}{x} - k' t\\
                               &= T_{2,n, 1-\alpha}(-x).
\end{aligned} 
\]
Thus, from $k' = -k$ in the above, one finds that for any $x$ in the interior of $\mathbb I_{2,n,\alpha}$ with $T_{2,n,\alpha}(x) = A^kC\cdot x \neq \ell_0(\alpha)$, one has $-x  \in \mathbb I_{2,n,1-\alpha}$ and $T_{2,n,1-\alpha}(-x) =  -T_{2,n,\alpha}(x) = C (A^k C )^{-1}C\cdot (-x)$.

\subsubsection{Planar maps}\label{sss:planarMaps}
From the previous, for $x$ as above and any $y$-value,  since when $m=2$ we have $R = C$,
\[
\begin{aligned} 
\mathcal T_{2,n,1-\alpha}(-x, -y) &= (- A^k C\cdot x,   R C (A^k C)^{-1} C R^{-1}\cdot -y)\\
                                                    &= (- A^k C\cdot x,   (A^k C)^{-1}   \cdot -y)\\
                                                    & (- A^k C\cdot x,  \frac{1}{k t + y}) \\
                                                    &= - ( A^k C\cdot x,  R    A^k C  R^{-1}\cdot y)   \\
                                                    &= - \mathcal T_{2,n,\alpha}(x, y).
\end{aligned} 
\]

Thus, since $d \mu = (1 + xy)^{-2} dx\, dy$ is preserved by 
\begin{equation}\label{eq:symS}
\mathcal S: (x,y) \mapsto (-x,-y),
\end{equation}
the bijectivity  up to $\mu$-null sets of  $\mathcal T_{2,n,\alpha}$ on a domain, such as is shown for $\Omega_{2,n,\alpha}$ below,  is equivalent to that of $\mathcal T_{2,n,1-\alpha}$ on $\mathcal S(\Omega_{2,n,\alpha})$.  Thus, the   symmetry announced in Theorem~\ref{t:mIsTwoSymm} will hold.

\subsection{Entropy of symmetric Rosen fractions}\label{ss:symmRosen}   Given the above symmetry, the value $\alpha=1/2$ has heightened interest.  We fill in missing details from \cite{ArnouxSchmidtCommCF}.     There,  a family of interval maps called the ``symmetric Rosen maps" are introduced and studied.   In fact, in our notation the symmetric Rosen map of index $n$ is $g = g_n = T_{2,n, 1/2}$.    For each $n \ge 3$, their definition begins with setting $\lambda_n = 2 \cos \pi/n$ and   $\mathbb I_n =   [-\lambda_n/2, \lambda_n/2)$.    Letting, as in \cite{NakadaFord}, $\lfloor  x \rfloor_n = a  \lambda_n $ such that $x -a \lambda_n \in  \mathbb I_n$, Rosen's \cite{Rosen} continued fraction  map on $\mathbb I_n$ is given by $f(x) = f_n(x)= |1/x| - \lfloor |1/x| \rfloor_n$ (with $x = 0$  fixed).  Similarly, \cite{ArnouxSchmidtCommCF}    defines   the  symmetric Rosen map $g(x)= g_n(x) = -1/x - \lfloor-1/x\rfloor_n$.   In  \cite{ArnouxSchmidtCommCF} this function is denoted by $h(x)$, and the indexing variable is $q$ instead of $n$.  Since $\lambda_n = t_{2,n}$, one easily verifies that $g_n = T_{2,n, 1/2}$.  

Of course,   $f(x) = g(x)$  if $x\le 0$.    On the other hand, for $x>0$,   we have $f(x) = g(-x)$.     As noted in \cite{ArnouxSchmidtCommCF}   $g(-x) = - g(x)$ unless $g(x) = -\lambda_n/2$.    Thus, for almost all $x>0$,   $g'(x) = (- g(-x))' = g'(-x) = f'(-x)$.    Hence,  $|g'(x)| = |f'(x)|$ on $\mathbb I_n$, up to a null set.

 A planar natural extension $\Omega_n$ is  given for  each $f_n$ in \cite{BKS}; for each,  $\forall (x,y) \in \Omega_n, \,  y\ge 0$.   Let $\mathcal G = \mathcal G_n$ be the planar map associated to $g$ in our usual manner.  We now justify the implicit claim of  \cite{ArnouxSchmidtCommCF} that   $\mathcal G$ is bijective up to $\mu$-null sets on the union   $\Omega_n \cup \mathcal S(\Omega_n)$.  Let $\mathcal F$ denote the two dimensional map related to $f = f_n$ in our usual manner.  We have $\mathcal G(x,y) = \mathcal F(x,y)$ for $x\le 0$ and $(x,y) \in \Omega_n$ and $\mathcal F(x,y) = \mathcal G(-x,-y)$ for $x>0$ and $(x,y) \in \Omega_n$. Therefore, up to null sets, $\mathcal G$ sends $\{x\le 0\} \cap \Omega_n$ bijectively to $\Omega_n$.   The symmetry  $\mathcal G\circ \mathcal S( x, y) = -\mathcal G(x,y)$ now shows that $\mathcal G$ is bijective up to null sets on $\Omega_n \cup \mathcal S(\Omega_n)$.    Note that in our notation, $\mathcal G = \mathcal T_{2,n,1/2}$  and one can verify that  $\Omega_n \cup \mathcal S(\Omega_n) = \Omega_{2,n,1/2}$
 
Since the measure $\mu$ is invariant under the symmetry $\mathcal S$, we have   $\mu(\Omega_n \cup \mathcal S(\Omega_n)) = 2 \mu(\Omega_n)$.     Applying the Rohlin formula \eqref{e:rohlinEnt}, the entropy of the symmetric Rosen map $g_n(x)$ satisfies

 \[\begin{aligned} 
h(g_n) \, \mu(\Omega_n \cup \mathcal S(\Omega_n))&=   \int_{\Omega_n \cup \mathcal S(\Omega_n)}\, \log |g'(x)|\, d\mu\\
 &= 2 \int_{\Omega_n}\, \log |g'(x)|\, d\mu\\
 \\
  &=   2 \int_{\Omega_n}\, \log |f'(x)|\, d\mu\,.\\
 \end{aligned}\]
 
Thus,  Nakada's result,  Theorem~\ref{t:HitoshiShowsGetHalf}, then gives  
\begin{equation}\label{e:symmRosenEntr}
 h(T_{2,n,1/2}) \, \mu( \Omega_{2,n,1/2}) = \text{vol}\, T^1(G_{2,n}\backslash \mathbb H).
 \end{equation}

 \subsection{Measure of $\Omega_n$ for Rosen fractions: Correcting an (typographical?) error in \cite{BKS}} \label{ss:Correction}   

  There is little doubt but that the author of this paper is responsible for an error in \cite{BKS}.   In short,  [\cite{BKS}, Lemma~3.4] is wrong as stated.   The error looks to be caused by what seems to be a typographical error in the formula for the $\mu$-mass of the planar extension there at hand.   We correct this latter in the following. 
 
 \bigskip
  
  The $\mu$-mass of the planar extension for the Rosen interval map of odd index $n\ge 3$ is given by 
 \begin{equation}\label{e:correctMuMassOddRosen}
  \mu(\Omega_n) = \ln \bigg(\dfrac{ 1 +R_n}{2 \sin \frac{\pi}{2 n}} \bigg), 
 \end{equation}
 where $R_n$ is the positive solution of $x^2 + (2 - 2 \cos \pi/n\,) x -1 = 0$.      Note that this expression for $\mu(\Omega_n)$ also holds for $n=3$, a case not considered in \cite{BKS}, where the Rosen continued fraction is the standard `nearest integer' continued fraction.
 
 \bigskip
The denominator in the corrected expression is the multiplicative inverse of the \cite{BKS} quantity  $B_{h+1} := \sin \frac{(h+1) \pi}{n}/\sin \frac{\pi}{n}$ with $n = 2 h + 3$.   The flaw in \cite{BKS} is either simply a dropping of this factor, or is based upon an error in the proof of [\cite{BKS}, Lemma~3.4] where a certain denominator is written as being $B_{j+2}$ but should be $B_{j+1}$.   With that corrected, a telescoping product has cancellation that leaves $B_{h+1}$ in place of simply $1$.

The formula can be compared with the (correct!)  formula when $n$ is even for $\mu(\Omega_n)$ given by    \cite{BKS}: $\mu(\Omega_{\text{even}\, n}) =   \ln [\frac{ 1 + \cos \pi/n}{  \sin  \pi/n}]$, see Figure~\ref{f:muMassRosenOm}.  
 \bigskip
 
\begin{figure} \centering
\scalebox{0.4}{
\includegraphics{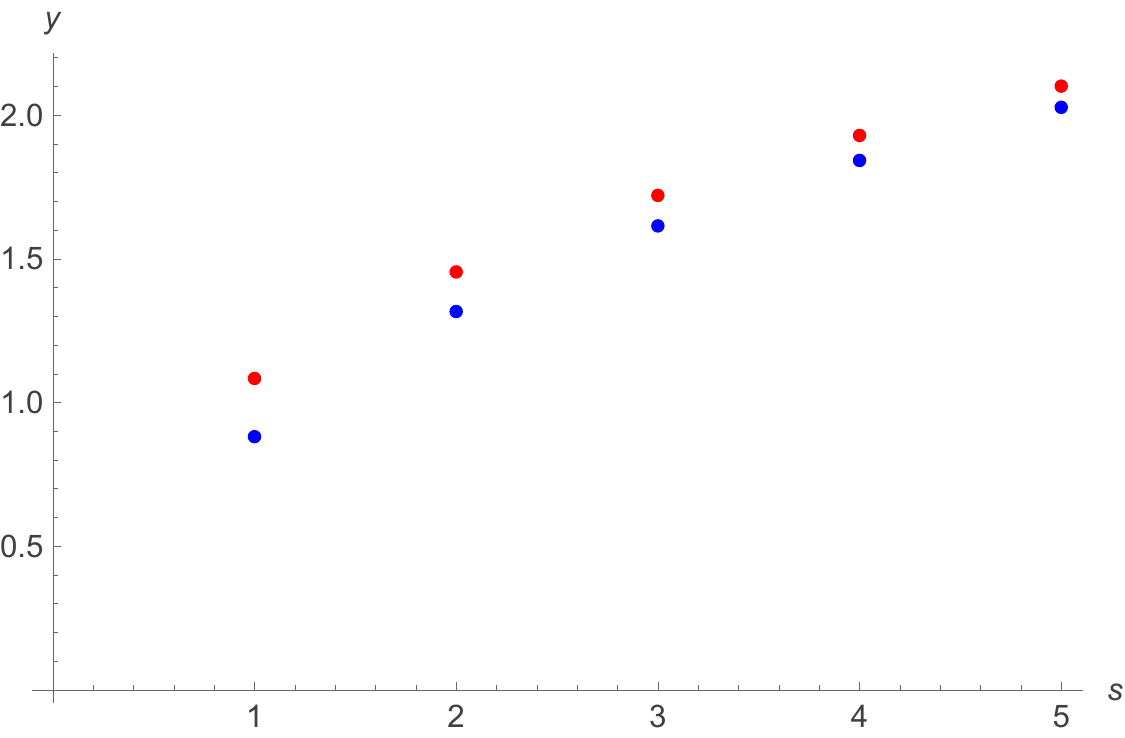}
}
\caption{The $\mu$-mass of the Rosen planar extension $\Omega_{2s+3}$ of \cite{BKS} (in red) is comparable to that of  $\Omega_{2s+2}$ (in blue), with both going to infinity as $s\to \infty$.  See \S~\ref{ss:Correction}.}
\label{f:muMassRosenOm}
\end{figure}

\subsection{Matching relation}\label{ss:matchingRel} Given $k$ and $v \in \mathcal V$, we have $R_{k,v}$ as in \eqref{e:rKv}.   We now follow \cite{CaltaKraaikampSchmidt} to define $L_{k,v}$.  With $m=2$, the word $w$ above becomes $w= (-1)^{n-3}, -2$.  Correspondingly we let  
\begin{equation}\label{eq:defW} W= W_{2,n} = A^{-2} C(A^{-1}C)^{n-3}
\end{equation}
and consider the following, each a word equal to $W^{k-1} (A^{-1}C)^{-1}$,  
\[\tilde{\mathcal C}_k = \begin{cases} W^{k-2}A^{-2}C(A^{-1}C)^{n-4}& \text{if}\; n>3, k>1\\
                                                            W^{k-2}A^{-1}& \text{if}\; n=3, k>1\\
                                                             (A^{-1}C)^{-1}&\text{if}\; n>3, k=1.
                                      \end{cases}
\]
For ease, let
\[ \tilde{\mathcal C} = \tilde{\mathcal C}_k  \; \text{and}\; \tilde{\mathcal D} = \tilde{\mathcal C}_{k+1}\]
and define 
\begin{equation}\label{eq:defL}
L_{k,v}  = (A^{-1} C)^{n-3} \; \tilde{\mathcal C}^{c_s}\tilde{\mathcal D}^{d_{s-1}} \cdots  \tilde{\mathcal D}^{d_1}  \tilde{\mathcal C}^{c_1-1}  W^{k-1}  A^{-1},                                     
\end{equation}
where when $k=1$ we simplify adjacent occurrences of  powers of $A^{-1} C$, which as we will see, in every case leads to only nonnegative powers appearing. 

\bigskip
\begin{Eg}\label{e:mIs3AndkIs2}   Consider $n = 3$ with $(k,v) = (2, p)$.  Thus,  $R_{k,v} = (A^2 C)^p$ and $L_{2,p} = (A^{-1} C)^{n-3} \; \tilde{\mathcal C}^{p-1}  W^{k-1}  A^{-1}=     (A^{-1})^{p-1} W   A^{-1} =  (A^{-1})^p \; A^{-2}C     A^{-1} =  A^{-(p+2)}C     A^{-1}$.      Similarly, still with  $n =  3$,  for $(k,v) = (2, 212)$  we find  
\[
\begin{aligned} 
L_{2,212} &= \tilde{\mathcal C}^2\;\tilde{\mathcal D}\;\tilde{\mathcal C} \; W A^{-1}\\
                 &=(A^{-1})^2  \;  A^{-2}CA^{-1} \; A^{-1}\; A^{-2}CA^{-1} \\
                 &=   A^{-4}C\;A^{-4}C\; A^{-1}  
\end{aligned}
\]

Let's now try $v = 121$.  
\[
\begin{aligned} 
L_{2,121} &= \tilde{\mathcal C}\;\tilde{\mathcal D}^2 \; W A^{-1}\\
                 &=A^{-1}  \;  (A^{-2}CA^{-1})\, (A^{-2}CA^{-1})  \;  A^{-2}CA^{-1} \\
                 &=   (A^{-3}C)^3\;A^{-1}.
\end{aligned}
\]

Let's now try the longer $v = 12121$.  
\[
\begin{aligned} 
L_{2,12121} &= \tilde{\mathcal C}\;\tilde{\mathcal D}^2 \; \tilde{\mathcal C}\;\tilde{\mathcal D}^2 \;W A^{-1}\\
                 &=A^{-1}  \;  (A^{-2}CA^{-1})\, (A^{-2}CA^{-1})  \;  A^{-1}  \;  (A^{-2}CA^{-1})\, (A^{-2}CA^{-1})  A^{-2}CA^{-1} \\
                 &=   (A^{-3}C)^2\;A^{-4}C\;(A^{-3}C)^2\;A^{-1}.
\end{aligned}
\]
\end{Eg}

\bigskip
\begin{Eg}\label{e:seeingInverseCausesNoProblem}   Consider $n =  7$ with $(k,v) = (1, 2)$.  We have $R_{k,v} = (A C)^2$ and $L_{1,2} = (A^{-1} C)^{n-3} \; \tilde{\mathcal C}^{c_1-1}  W^{k-1}  A^{-1}= (A^{-1} C)^4 \;  (A^{-1} C)^{-1}     A^{-1} = (A^{-1} C)^{3}     A^{-1}$.     Similarly, still with  $n =  7$,  for $(k,v) = (1, 212)$  we find  
\[
\begin{aligned} 
L_{1,212} &= (A^{-1} C)^{n-3} \;  \tilde{\mathcal C}^2\tilde{\mathcal D}\tilde{\mathcal C}   A^{-1}\\
                 &= (A^{-1} C)^4 \;  (A^{-1} C)^{-2}  \; A^{-2}C(A^{-1}C)^{n-4}\;  (A^{-1} C)^{-1}\; A^{-1} \\
                 &=  (A^{-1} C)^2 \;  A^{-2}C(A^{-1}C)^{2}\; \; A^{-1}  
\end{aligned}
\]
Thus, with $n=7$ we have $\underline{S}(1,212) = \overline{S}(212)$. 
\end{Eg}

Compare the following with [\cite{CaltaKraaikampSchmidt}, Lemma~5.1].    Fix $n$ and take $R_{k,v}$ as in \eqref{e:rKv}.  
\begin{Lem}\label{l:shortRightId}   The following identities in $G_{2,n}$ hold. 
\begin{enumerate} 
\item[(i)]\quad  $A^k C A =  C A  W^k$  for each $k\in \mathbb Z$;

\medskip
\item[(ii)]\quad   $(C^{-1}AC) (A^kC)^a (C^{-1}AC) ^{-1} = (A^{-1}C)^{n-3}\, \tilde{\mathcal C}_{k}^{\, \mathstrut a} \,(A^{-1}C)^{-(n-3)}$ for all $a, k \in \mathbb Z$;

\medskip
\item[(iii)]\quad  $\tilde{\mathcal C}_k(A^{-1}C)^{-(n-3)} (C^{-1}AC)  =W^{k-1} A^{-1}$;

\medskip
\item[(iv)]\quad   $(C^{-1}AC) R_{k,v}  =  L_{k,v}\,$.
\end{enumerate} 
\end{Lem} 
\begin{proof}    We have  $CA \, W \, (CA)^{-1} = CA\,  A^{-2} C(A^{-1}C)^{n-3} \,(CA)^{-1} = (C A^{-1})^{n-1} C^{-1} = A C^{-1} C^{-1} = A$, as $C$ has projective order two.    Thus for any $k$,  we have $CA \, W^k \, (CA)^{-1} = A^k$ from which (i) holds.

Now, $(C^{-1}AC)  A^kC =  (C^{-1}AC)   C A  W^k A^{-1} = C^{-1}A^2 W^k A^{-1}= (A^{-1}C)^{n-3} W^{k-1} A^{-1}$. 
Since  $\tilde{\mathcal C}_k$  equals $W^{k-1}(A^{-1}C)^{-1}$, one has  $(C^{-1}AC)  A^kC =  (A^{-1}C)^{n-3}  \tilde{\mathcal C}_k (A^{-1}C) A^{-1}$. Since $(A^{-1}C)^{n-3}= (A^{-1}C)^3 = (A^{-1}C) A^{-1} (C^{-1}AC)^{-1}$, the second statement holds. 

Since $(A^{-1}C)^{-(n-3)} (C^{-1}AC)   =  (A^{-1}C)^3 C^{-1}A C =  (A^{-1}C)^2 C = A^{-1} C A^{-1}$, 
we find  $W^{k-1} A^{-1} = W^{k-1}(A^{-1}C)^{-1} \;   A^{-1} C A^{-1}  =\tilde{\mathcal C}_k(A^{-1}C)^{-(n-3)} (C^{-1}AC)$.  Thus, the third statement holds. 

From (ii)   and (iii), 
\[(C^{-1}AC) (A^kC)^{c_1}  = (A^{-1}C)^{n-3}\, \tilde{\mathcal C}_{k}^{\, \mathstrut c_1-1} \, \tilde{\mathcal C}_{k}\,(A^{-1}C)^{-(n-3)}(C^{-1}AC) = (A^{-1}C)^{n-3}\, \tilde{\mathcal C}_{k}^{\, \mathstrut c_1-1} W^{k-1} A^{-1}. \]
Thus (iv) holds when $s=1$.  When $s>1$, repeated applications of (ii)  give  
\[(C^{-1}AC) R_{k,v}  =  (A^{-1} C)^{n-3} \; \tilde{\mathcal C}^{c_s}\tilde{\mathcal D}^{d_{s-1}} \cdots  \tilde{\mathcal D}^{d_1} (A^{-1}C)^{-(n-3)} (C^{-1}AC) \tilde{\mathcal C}^{c_1}.\]
Thus, the previous displayed equation applies to  show that here also (iv) holds.
 
\end{proof}

 \subsection{Initial digits of right endpoints}  We recycle the notation of \S~\ref{ss:CKS}  to the setting here of  $m=2$.    Fix  $n\ge 3$ and set $t = t_{2,n}$.  Since $AC\cdot x = t-1/x$ (here $A$ now denotes $A_{2,n}$ and $C$ denotes $C_2$),   an elementary calculation shows that given $\zeta_{1,1} = 1 - \sqrt{1 - 1/t^2}$.    The largest value of $x$ we contemplate is $x=t$; since $AC\cdot t = t^2-1/t$ is less than or equal to $t$ for all $n \ge 3$, we find that $I_{1,1} = [ 1 - \sqrt{1 - 1/t^2}, 1]$ for $n>3$.  Note that this value, $\zeta_{1,1}$ is less than $1/2$.   When $n=3$ one finds that exactly when $\alpha=1$ does $T_{2,3,\alpha}(x)$ ever equal $AC\cdot x$ and in fact this equality holds only for $x=1$.    A similar direct calculation shows that for $n=3$,  we have $I_{2,1}$ extending from $\zeta_{2,1}$ to $1$ and that $\zeta_{2,1}$ is less than $1/2$.   In all cases,  moving to the left from this rightmost subinterval of the parameters,  one finds for all possible greater values of $k$ that  $r_1(\alpha) = A^kC\cdot r_0(\alpha)$ does appear.  Thus, the parameter intervals $0< \alpha < 1$ are partitioned by the subintervals $I_{k,1}$ indexed by the positive $k$, although (by a minor abuse) $I_{1,1} = \emptyset$ when $m=3$.

   We now say more about the orbit of points in $I_{1,1}$ when $n>3$.   Upon setting $m=2$ the element of [\cite{CaltaKraaikampSchmidt}, Eq. 3.1] becomes  $U = A^2C (AC)^{n-3}$; the final lines of the proof of [\cite{CaltaKraaikampSchmidt}, Prop. 3.1] show that $U\cdot t = t$ (with an argument that also holds in our case of $m=2$).  From this, one can show that for  $n>3$ when $k=1$ the only possible $v$ are limited to have $c_i \le n-3$.  Indeed, for any $x<0$ we have $AC \cdot x = t-1/x >t\ge r_0(\alpha)$ and hence already $AC$ never appears for any negative $x$.  A direct calculation with $0<x<t$ gives  $(AC)^{n-2}\cdot x = (C A^{-1})^2\cdot x = -1/-t-1/(-t+ x) = (x-t)/((x-t)t+1)$ and this is easily shown to be less than $x-t$.  This last signals that  $(AC)^{n-2}$ is inadmissible.   That is, when $n>3$ the digit $k=1$ can occur, but the only admissible $R_{1,v}, v = c_1 d_1 \dots d_{s-1} c_s$ are such that  $c_i<n-2$ for all $i$.

\subsection{Matching intervals} 
 The description of the matching intervals when $m=3$ of  [\cite{CaltaKraaikampSchmidt}, Theorem~4.12] needs only mild adjustment to give the following.     We use the notation of \S~\ref{ss:ExpanSynTree}, with the interpretation throughout that elements are determined in terms of $(m,n) = (2,n)$.

\begin{Thm}\label{t:partition}  Fix $m=2$ and $n\ge 3$.   Then  
\[(0, 1) = \bigcup_{k=1}^{\infty}\, I_{k,1}\,.\]

If $n=3$ then $I_{1,1} = \emptyset$.   
Furthermore, for each $k \in \mathbb N$ and each $v \in \mathcal V$, one has
\[ I_{k,v} = J_{k,v} \cup \, \bigcup_{q=q_0}^{\infty}\, I_{k, \Theta_q(v)}\,,\]
where $q_0 = 0$ unless $v= c_1$, in which case $q_0= -1$.    Moreover, when $n>3$ one has  $I_{1,v} = \emptyset$ for all $v= c_1 \dots$ with $c_1>n-3$.
 \end{Thm} 
The proof of the analog of the theorem in the case of $m=3$ is given in  [\cite{CaltaKraaikampSchmidt}, Subsections 4.3 and 4.4].  This proof succeeds {\em mutatis mutandis} for the main cases above.    The special restrictions in our setting of $m=2$ are a consequence of the discussion in the previous subsection.  

\subsection{Planar extensions along matching intervals}\label{ss:PlanarExtMatch}

\subsubsection{Planar extensions $\Omega_{2,n, \alpha}$ for $\alpha$ interior to a matching interval}   In  [\cite{CaltaKraaikampSchmidtContinEntrop}, Table~1]    one is pointed to the appropriate definitions  of what we denote as $\Omega_{3,n, \alpha}$.    Those definitions  for  their case of `small $\alpha$' apply directly to the case of $m=2$, giving  the various  $\Omega_{2,n, \alpha}$ with $\alpha \in (0,1)$.   A main reason that the constructions and arguments of \cite{CaltaKraaikampSchmidtContinEntrop} succeed even when $m=2$ is that the key lemma, [\cite{CaltaKraaikampSchmidtContinEntrop}, Lemma~14] is based on the fact that matrices $A^p C, p \in \mathbb Z$ are of the shape  $\begin{pmatrix} a&b\\-b& 0\end{pmatrix}$.   Since that is also true when $m=2$,  for 
real $x$, we have 
\begin{equation}\label{e:yActionDigitsRelation}
(R A^pC R^{-1})^{-1}\cdot (-x) = - \; (A^pC\cdot x).
\end{equation}
   With this, we can pass from understanding of sequences of digits to actions of powers of our $\mathcal T_{2,n,\alpha}$.

 Fixing $(k,v)$, recall that $\underline{S} = \underline{S}(k,v)$ denotes the length of $\underline{d}(k,v)$ as a word in $\{-1, -2\}$.  For $\alpha$ interior to a matching interval $J_{k,v}$,   let 
  \[\Omega^+ =  ( [\ell_{\underline{S}+1}(\alpha), r_0(\alpha)) \times [0,  -\ell_{\underline{S} - i_{\underline{S}+1}}(\zeta_{k,v}]) \cup  \bigcup_{a= 1}^{\underline{S}+1}\, ( [\ell_{i_a}(\alpha), \ell_{i_{a+1}}(\alpha)) \times [0,  -\ell_{\underline{S} - i_a}(\zeta_{k,v}]), \] where 
 $\ell_{i_1}(\alpha), \dots,  \ell_{i_{\underline{S}+1}}(\alpha)$ gives the initial $T_{2,n,\alpha}$-orbit of $\ell_0(\alpha)$ in increasing order as real numbers, and similarly for the $\ell_{i_a}(\zeta_{k,v})$.  It is helpful to define the values $\tau(i)$  by $\tau(j) = a$ exactly when $i_a = j$.      Thus,   $y_{\tau(i)} = - \ell_{\underline{S}-i}(\zeta_{k,v})$.    That the $y_i$ are increasing is shown in [\cite{CaltaKraaikampSchmidtContinEntrop}, Lemma~35].  Hence, $y_{\underline{S}+1} = - \ell_0(\zeta_{k,v})$.

       Similarly, let 
 \[ \Omega^{-} = ([\ell_0(\alpha), r_{\overline{S}}(\alpha)] \times  [- r_{j_{-1}}(\eta_{k,v}), 0]) \cup  \bigcup_{b=-1}^{ -\overline{S}}\,    ([ r_{j_{b-1}}(\alpha), r_{j_b}(\alpha)) \times [- r_{\overline{S}+b}(\eta_{k,v}), 0]),   \]
where now $r_{j_{-1}}(\alpha), r_{j_{-2}}(\alpha),\dots, r_{j_{-\overline{S}-1}}(\alpha)$  denotes the initial $T_{2,n,\alpha}$-orbit of $r_0(\alpha)$ in decreasing order, and similarly for the $r_{j_b}(\eta_{k,v})$ and $\overline{S} = \overline{S}(v) $ is the length of $\overline{d}(k,v)$ as a word in $\{k, k+1\}$.   Again a result in \cite{CaltaKraaikampSchmidtContinEntrop} shows that   $y_{-1}> y_{-2} > \cdots >y_{-\overline{S}-1}$.  Note that $y_{-\overline{S}-1} = - r_0(\eta_{k,v})$.

 The arguments of [\cite{CaltaKraaikampSchmidtContinEntrop}, Section~5] show that $\mathcal T_{3,n, \alpha}$ is bijective on the domain $\Omega_{3,n, \alpha}$ up to null sets (and thus our two descriptions of $\Omega_{3,n,\alpha}$ agree).   Those arguments discuss the portions of the domain which fiber over the respective cylinders of the interval map, and their images.  As Figure~\ref{f:mIs2NIs2AlphaIsPt2} hints (in an albeit particularly simple case)  these arguments succeed for $m=2$ and any $\alpha$ interior to matching intervals.   This is due, besides the aforementioned success of  [\cite{CaltaKraaikampSchmidtContinEntrop}, Lemma~14],  to the fact that those arguments never need to consider the order of $C$, but are rather based on the geometry of the actions of $A$ and $C$.    

\subsubsection{An example:  $\Omega_{2,3,1/5}$}\label{sss:OneFifth}  Our $\Omega_{2,3,1/5}$, see  Figure~\ref{f:mIs2NIs2AlphaIsPt2},  must agree  with the attractor for $\alpha= 1/5$ shown in [\cite{CarminatiIsolaTiozzo}, Figure~8].  That this is so,  up  to a change of coordinates (due to different normalizing choices), the reader can easily verify from the following detailed information.   

  With $m=2, n= 3$, one has $t = t_{2,3} = 1$ and    the equality $A^5 C \cdot 1/5 = 0$ is trivially checked.   Thus,   $\alpha = 1/5$ belongs to the matching interval $J_{k=5, v=1}$.   The left endpoint of this matching interval is $\zeta = 3 - 2 \sqrt{2}\sim 0.172$, as   verified by  $A^5 C \cdot \zeta = \zeta -1$.   Since $v = c_1=1$,    \eqref{eq:defL} gives  $L_{5,1} A = W^4 =  (A^{-2} C)^4$.   From the definitions in the previous subsubsection,  the upper horizontal boundaries of $\Omega_{2,3,1/5}$ have $y$-coordinates   $\{-\ell_0(\zeta),   -\ell_1(\zeta),  -\ell_2(\zeta),  -\ell_3(\zeta),  -\ell_4(\zeta)\} \sim \{ 0.828,  0.793,  0.739,  0.646,  0.453\}$.     
   Since $\eta = (\sqrt{2}-1)/2 \sim 0.207$ satisfies $C^{-1} A C A^5 C\cdot \eta = \eta + 1$, this is the right endpoint of the matching interval.      Thus the lower horizontal boundaries of $\Omega_{2,3,1/5}$ are $\{-r_0(\eta),   -r_1(\eta)\}$ of approximate values  $\{ -0.207, -0.172\}$.

%
\def\reta{0.207}%
\def\rzeta{0.172}%
\def\lGoverSqrtFive{0.276393}%
\def\litG{0.618034}%
\def\lGSqrd{0.381966}%
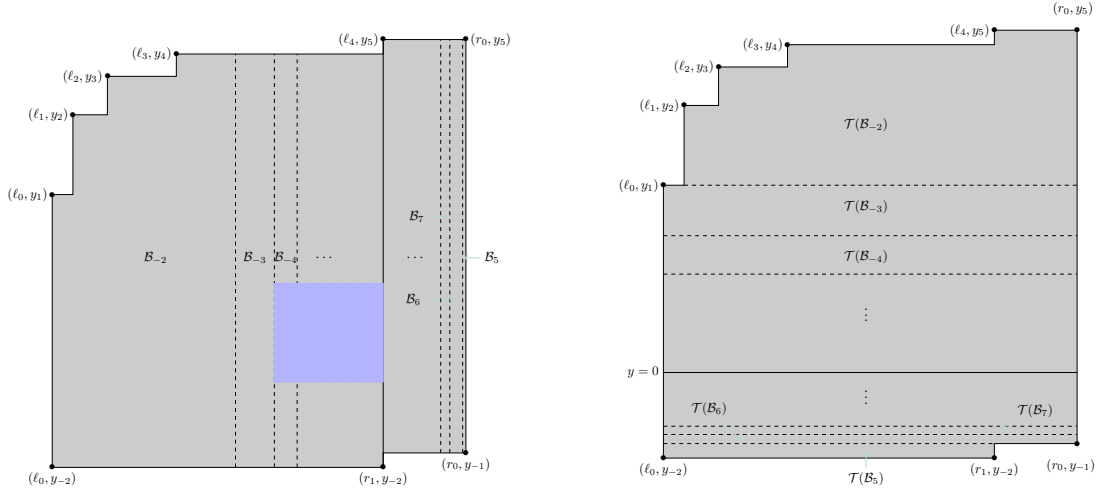
\begin{figure}[h]
\scalebox{.5}{
\noindent
\begin{tabular}{lll}
\begin{tikzpicture}[scale=11,fill=black!20]
\filldraw(-0.8,-\reta)--(0,-\reta)--(0,-\rzeta)--(0.2, -\rzeta)--(0.2, 0.828)--(0, 0.828)--(0, 0.793)--(-1/2, 0.793)--(-1/2, 0.739)--(-2/3, 0.739)--(-2/3, 0.646)--(-3/4, 0.646)--(-3/4, 0.453)--(-0.8, 0.453)--cycle;
 \foreach \x/\y in {-0.8/-\reta, 0/-\reta, 0.2/-\rzeta, 0.2/0.828, 0/0.828, -0.5/0.793,-0.666/0.739,-0.75/0.646,-0.8/0.453%
} { \node at (\x,\y) {$\bullet$}; } 
\draw (0, 0.828)--(0,-\reta);
\node[left] at (-0.8, 0.453) {$(\ell_0, y_1)$};
\node[left] at (-3/4, 0.646) {$(\ell_1, y_2)$};
\node[left] at (-2/3, 0.739) {$(\ell_2, y_3)$};
\node[left] at (-1/2, 0.793) {$(\ell_3, y_4)$};
\node[left] at (0, 0.828) {$(\ell_4, y_5)$};
\node[right] at (0.2, 0.828) {$(r_0, y_5)$};
\node[below] at (-0.8,-\reta) {$(\ell_0, y_{-2})$};
\node[below] at (0,-\reta) {$(r_1, y_{-2})$};
\node[below] at (0.2,-\rzeta) {$(r_0, y_{-1})$};
 %
\draw[dashed] (-0.357,-\reta)--(-0.357, 0.793);  
\node at (-0.55,0.3) {$\mathcal B_{-2}$};
\draw[dashed] (-0.263,-\reta)--(-0.263, 0.793);  
\node at (-0.31,0.3) {$\mathcal B_{-3}$};
\draw[dashed] (-0.208,-\reta)--(-0.208, 0.793);  
\node at (-0.235,0.3) {$\mathcal B_{-4}$};
\node at (-0.14,0.3) {\large{$\cdots$}};
 \draw[dashed] (0.192,-\rzeta)--(0.192, 0.828);  
\node at (0.185, 0.3)[pin={[pin edge=<-, pin distance=12pt]0:{$\mathcal B_{5}$}}] {};
 \draw[dashed] (0.161,-\rzeta)--(0.161, 0.828);  
\node at (0.19, 0.2)[pin={[pin edge=<-, pin distance=24pt]180:{$\mathcal B_{6}$}}] {};
 \draw[dashed] (0.139,-\rzeta)--(0.139, 0.828);  
\node at (0.16, 0.4)[pin={[pin edge=<-, pin distance=12pt]180:{$\mathcal B_{7}$}}] {};
\node at (0.08,0.3) {\large{$\cdots$}};
\draw[dotted, blue] (-0.263,0.238)--(0, 0.238);  
\draw[dotted, blue] (-0.263,0)--(0, 0);  
\filldraw[blue!30] (-0.263,0.238)--(0, 0.238)--(0, 0)--(-0.263,0)--cycle;
\end{tikzpicture}
&\phantom{a long arrow}
&
\begin{tikzpicture}[scale=11,fill=black!20]  
\filldraw(-0.8,-\reta)--(0,-\reta)--(0,-\rzeta)--(0.2, -\rzeta)--(0.2, 0.828)--(0, 0.828)--(0, 0.793)--(-1/2, 0.793)--(-1/2, 0.739)--(-2/3, 0.739)--(-2/3, 0.646)--(-3/4, 0.646)--(-3/4, 0.453)--(-0.8, 0.453)--cycle;
 \foreach \x/\y in {-0.8/-\reta, 0/-\reta, 0.2/-\rzeta, 0.2/0.828, 0/0.828, -0.5/0.793,-0.666/0.739,-0.75/0.646,-0.8/0.453%
} { \node at (\x,\y) {$\bullet$}; } 
\draw (-0.8, 0)--(0.2, 0); 
\node[left] at (-0.8, 0) {$y=0$};
\node[left] at (-0.8, 0.453) {$(\ell_0, y_1)$};
\node[left] at (-3/4, 0.646) {$(\ell_1, y_2)$};
\node[left] at (-2/3, 0.739) {$(\ell_2, y_3)$};
\node[left] at (-1/2, 0.793) {$(\ell_3, y_4)$};
\node[left] at (0, 0.828) {$(\ell_4, y_5)$};
\node[right] at (0.13, 0.88) {$(r_0, y_5)$};
\node[below] at (-0.8,-\reta) {$(\ell_0, y_{-2})$};
\node[below] at (0,-\reta) {$(r_1, y_{-2})$};
\node[below] at (0.185,-\rzeta-0.025) {$(r_0, y_{-1})$};
 %
\draw[dashed] (-0.75,0.453)--(0.2, 0.453);  
\node at (-0.31,0.6) {$\mathcal T(\mathcal B_{-2})$};
\draw[dashed] (-0.8,0.331)--(0.2, 0.331);  
\node at (-0.31,0.4) {$\mathcal T(\mathcal B_{-3})$};
\draw[dashed] (-0.8,0.238)--(0.2, 0.238);  
\node at (-0.31,0.28) {$\mathcal T(\mathcal B_{-4})$};
\node at (-0.31,0.15) {\large{$\vdots$}};
\draw[dashed] (-0.8,-\rzeta)--(0, -\rzeta);  
\node at (-0.31,-0.18)[pin={[pin edge=<-, pin distance=12pt]270:{$\mathcal T(\mathcal B_{5})$}}] {};
\draw[dashed] (-0.8,-0.15)--(0.2, -0.15);   
\node at (-0.6, -0.175)[pin={[pin edge=<-, pin distance=18pt]120:{$\mathcal T(\mathcal B_{6})$}}] {};
\draw[dashed] (-0.8,-0.13)--(0.2, -0.13);  
\node at (0, -0.145)[pin={[pin edge=<-, pin distance=12pt]30:{$\mathcal T(\mathcal B_{7})$}}] {};
\node at (-0.31,-0.05) {\large{$\vdots$}};
\end{tikzpicture}
\end{tabular}
}
\caption{On the left: The planar domain $\Omega_{2,3,1/5}$, and its partitioning `blocks'; each $\mathcal B_i$ fibers over the cylinder where $T_{2,3,1/5}$ is given by $x \mapsto A^iC\cdot x$.   The $T_{2,3,1/5}$-orbit of $r_0(1/5) = 1/5$ is labeled by $r_0, r_1$; the $T_{2,3,1/5}$-orbit of $\ell_0(1/5) = -4/5$ is labeled by $\ell_0,\ell_1, \cdots $.  One has $r_2(\alpha) = \ell_5(\alpha)$ for  other $\alpha$ in the matching interval $J_{k,v} = J_{5,1}$, but here the special equality arises:   $r_1 = \ell_4 = 0$.    The $y$-values of the horizontal boundaries are derived from appropriate $T_{2,3,\eta_{5,1}}$-  and  $T_{2,3,\eta_{5,1}}$-orbits, see the discussion of \S~\ref{sss:OneFifth}.  The blue rectangle indicates the region $\Omega^{+}_{\alpha, d}$ of   Lemma~\ref{l:firstReturnTo2ndQuadrant}.    On the right:  The image of $\Omega_{2,3,1/5}$ under $\mathcal T_{2,3,1/5}$ is $\Omega_{2,3,1/5}$ itself.}
\label{f:mIs2NIs2AlphaIsPt2}
\end{figure}

\subsection{Increasing, decreasing, constancy  of $\alpha \mapsto \mu(\Omega_{2,n,\alpha})$ on matching intervals}\label{ss:incrDecr}    The case of $n=3$ is treated in \cite{CarminatiIsolaTiozzo}.   We thus assume $n>3$.       We determine the behavior of $\alpha \mapsto \mu(\Omega_{2,n,\alpha})$ in our setting by describing the elementary geometry of the passage from  $\Omega_{2, n, \alpha}$ to $\Omega_{2, n, \alpha'}$ for $\alpha, \alpha'$ in the same matching interval.  In fact,  when $\alpha, \alpha'$ are  close neighbors as discussed in \cite{CaltaKraaikampSchmidtPfsErgodicity, CaltaKraaikampSchmidtContinEntrop} then a  more refined result can be achieved by using the  technique of `quilting', see [\cite{CaltaKraaikampSchmidtPfsErgodicity}, Main Theorem~2]

%
\def\reta{0.207}%
\def\rzeta{0.172}%
\def\lGoverSqrtFive{0.276393}%
\def\bigG{1.618034}%
\def\litG{0.618034}%
\def\lGSqrd{0.381966}%
\begin{figure}[h]
\scalebox{1}{
\begin{tikzpicture}[scale=7]
\draw(-0.8,-\reta)--(0,-\reta)--(0,-\rzeta)--(0.18, -\rzeta)--(0.18, 0.828)--(0, 0.828)--(0, 0.793)--(-1/2, 0.793)--(-1/2, 0.739)--(-2/3, 0.739)--(-2/3, 0.646)--(-3/4, 0.646)--(-3/4, 0.453)--(-0.8, 0.453)--cycle;
 \foreach \x/\y in 
 {-0.8/-\reta, -0.56/-\reta, 0/-\reta, 0.2/-\rzeta, 0.18/0.828, -0.36/0.828, -0.61/0.793,-0.72/0.739,-0.78/0.646,-0.667/ 0.646,-0.82/0.453%
} { \node at (\x,\y) {$\bullet$}; } 
\draw (0, 0.828)--(0,-\reta);
\node[shift={(-0.7,0)}] at (-0.82, 0.453) {$(\ell_0, y_1)$};
\node[left, shift={(0,-0.07)}] at (-0.78,0.646) {$(\ell_1, y_2)$};
\node[right, shift={(0,-0.2)}] at (-2/3, 0.646) {$(\ell'_2, y_2)$};
\node[left, shift={(0,0.1)}] at (-0.72, 0.739) {$(\ell_2, y_3)$};
\node[left, shift={(0,0.3)}] at (-0.61,0.793) {$(\ell_3, y_4)$};
\node[above] at (-0.36, 0.828) {$(\ell_4, y_5)$};
\node[above] at (0.18, 0.828) {$(r_0, y_5)$};
\node[below, shift= {(0, -0.1)}] at (-0.8,-\reta) {$(\ell'_0, y_{-2})$};
\node[below, shift= {(0, -0.1)}] at (-0.56,-\reta) {$(r_1, y_{-2})$};
\node[below, shift= {(0, -0.1)}] at (0,-\reta) {$(r'_1, y_{-2})$};
\node[right] at (0.2,-\rzeta) {$(r'_0, y_{-1})$};
%
\draw[pattern=north east lines, pattern color=blue]  (0.18,-\rzeta)--(0.2,-\rzeta) --(0.2,0.828) --(0.18,0.828)-- cycle;
\draw[pattern=north east lines, pattern color=blue]  (-0.56,-\reta)--(0,-\reta) --(0,-\rzeta) --(-0.56,-\rzeta)-- cycle;
\node at (0.2,  0.3)[pin={[pin edge=<-, pin distance=12pt]0:{$\mathcal A = \mathcal T_{A}(\mathcal D)$}}] {};
\node at (-0.3,  -0.2)[pin={[pin edge=<-, pin distance=16pt]270:{$\mathcal T_{\alpha'}(\mathcal A)$}}] {};
\draw[pattern=north west lines, pattern color=red]  (-0.82,-\reta)--(-0.8,-\reta) --(-0.8,0.453) --(-0.82,0.453)-- cycle; 
\draw[pattern=north west lines, pattern color=red]  (-0.78,0.453)--(-3/4,0.453) --(-3/4,0.646) --(-0.78,0.646)-- cycle; 
\draw[pattern=north west lines, pattern color=red]  (-0.72,0.646)--(-2/3,0.646) --(-2/3,0.739) --(-0.72,0.739)-- cycle; 
\draw[pattern=north west lines, pattern color=red]  (-0.61,0.739)--(-1/2,0.739) --(-1/2,0.793) --(-0.61,0.793)-- cycle; 
\draw[pattern=north west lines, pattern color=red]  (-0.36,0.793)--(0,0.793) --(0,0.828) --(-0.36,0.828)-- cycle; 
\node at (-0.83,  0.1)[pin={[pin edge=<-, pin distance=12pt]180:{$\mathcal D$}}] {};
\node at (-0.8,  0.55)[pin={[pin edge=<-, pin distance=40pt]180:{$\mathcal T_{\alpha}(\mathcal D)$}}] {};
\node at (-0.73,  0.69)[pin={[pin edge=<-, pin distance=45pt]180:{$\mathcal T_{\alpha}^{2}(\mathcal D)$}}] {};
\node at (-0.55,  0.8)[pin={[pin edge=<-, pin distance=16pt]90:{$\mathcal T_{\alpha}^{3}(\mathcal D)$}}] {};
\node at (-0.14,  0.83)[pin={[pin edge=<-, pin distance=16pt]90:{$\mathcal T_{\alpha}^{4}(\mathcal D)$}}] {};
\end{tikzpicture}
}
\caption{{\bf Decreasing mass along $J_{1,5}$.}   With $m=2, n=3$,  both $\alpha = 0.18$ and  $\alpha' = 0.2$  lie in the matching interval $J_{1, 5}$.   Along any matching interval, the coordinates of all tops and bottoms of the rectangles which comprise the various $\Omega_{2,n,\alpha}$ are constant  (here $y_1, \dots y_5$ and $y_{-2}, y_{-1}$).   The function $\mathcal T_A$ maps the rectangle $\mathcal D = [\ell_0, \ell'_0)\times [r_{-\overline{S}-1}, y_1]$ to $\mathcal A = [r_0, r'_0)\times [r_{-1}, y_{\overline{S}+1}]$, where here we use simplified notation of $\ell_0$ for $\ell_0(\alpha)$ and $\ell'_0$ for $\ell'_0(\alpha)$, and similarly for $r_0, r'_)$.  To pass from $\Omega_{2,n,\alpha}$ to $\Omega_{2,n,\alpha'}$,  we delete the initial $\mathcal T_{2,n,\alpha}$-orbit of  $\mathcal D$ and add $\mathcal A$ and its initial $\mathcal T_{2,n,\alpha}$-orbit.   Since each step preserves $\mu$-measure, the sign of the change in mass is that of  $(\overline{S}(k) - \underline{S}(k,v))$. See Lemma~\ref{c:massIncreasesWhere}.}%
\label{f:changeOfMass}%
\end{figure}
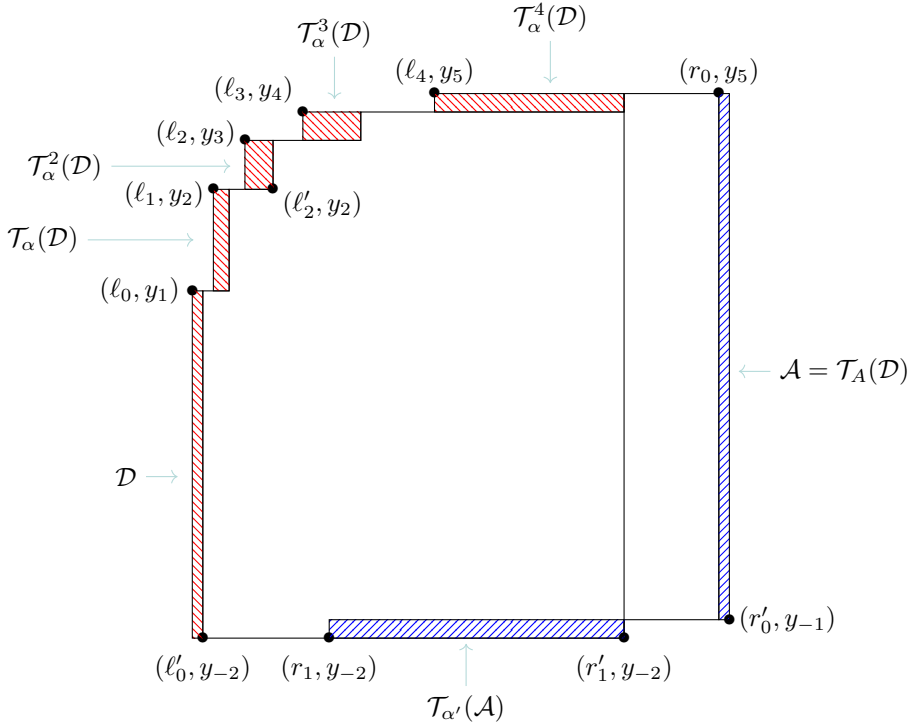

 Recall that given a matching interval $J_{k,v}$ that $\underline{S}(k,v)$ denotes the length of $\underline{d}(k,v)$ as a word in $\{-1, -2\}$, and $\overline{S}(v)$ is the length of $\overline{d}(k,v)$ as a word in $\{k, k+1\}$.  
 
\begin{Lem}\label{l:massIncreasesWhere}    Fix $n> 3$.  The function   $\alpha \mapsto  \mu(\Omega_{2,n,\alpha})$ increases on any matching interval $J_{k,v}$ for which $\underline{S}(k,v) < \overline{S}(v)$.   The function decreases on $J_{k,v}$ for which the inequality is in the other direction, and is constant when $\underline{S}(k,v) = \overline{S}(v)$.
\end{Lem}
\begin{proof}    For $\alpha \in J_{k,v}$, the domain $\Omega_{2,n,\alpha}$ is the union of rectangles whose heights are depend only on the maps $T_{2,n, \zeta_{k,v}}$ and  $T_{2,n, \eta_{k,v}}$ whereas the widths of these rectangles are the distance between nearest elements of 
 $\{\ell_i(\alpha)\,|\, 0\le i \le \underline{S}(k,v)\}\cup \{r_j(\alpha)\,|\, 0\le j \le \overline{S}(v)\}$.  Along the $ J_{k,v}$ the ordering of this set remains unaltered.   
 
 Given $\alpha < \alpha'$ both in $J_{k,v}$ we begin with $\Omega_{2,n,\alpha}$.  Let $\mathcal D = [\ell_0, \ell'_0)\times [r_{-\overline{S}(v)-1}, y_1]$,  with simplified notation of $\ell_0$ for $\ell_0(\alpha)$ and $\ell'_0$ for $\ell'_0(\alpha)$.    Also let $\mathcal A = T_A(\mathcal D)$.  
 
 We claim that 
 \[ \Omega_{2,n,\alpha'} = \bigg( \Omega_{2,n,\alpha} \setminus \bigcup_{i=0}^{\underline{S}(k,v)}\, \mathcal T_{2,n,\alpha}^i(\mathcal D)\bigg) \cup \bigcup_{i=0}^{\overline{S}(v)}\, \mathcal T_{2,n,\alpha'}^i(\mathcal A)\big),\]
 see Figure~\ref{f:changeOfMass}.
Since $\mathcal T_A$ and each application of either $\mathcal T_{2,n,\alpha}$ or $\mathcal T_{2,n,\alpha'}$ preserves $\mu$-measure, the  claim  implies  $\mu(\Omega_{2,n,\alpha'}) = \mu(\Omega_{2,n,\alpha}) + (\,\overline{S}(v) - \underline{S}(k,v)\,) \, \mu(\mathcal D)$, whereby the result holds. 
 
 We now show that the claim does hold.

 \smallskip
 \noindent
 {\it Step 1: Correct deletion.}   Since both $\alpha, \alpha' \in J_{k,v}$, the $T_{2,n,\alpha}$-digits of $\ell_0(\alpha)$ through $\ell_{ \underline{S}(k,v)-1}(\alpha)$ agree with the $T_{2,n,\alpha'}$-digits of $\ell_0(\alpha')$ through $\ell_{\underline{S}(k,v)-1}(\alpha')$ and similarly for the respective digits of $r_0, \dots, r_{\overline{S}(v)-1}$ and $r'_0, \dots, r'_{\overline{S}(v)-1}$.     For $0\le i \le \underline{S}(k,v)-1$ let $E_i$ be the suffix of $L_{k,v}A$ of length $i$.    With this,  $\ell_i = T_{2,n,\alpha}^i(\ell_0) = E_i\cdot \ell_0$.    
 
From \eqref{e:yActionDigitsRelation} 
it follows that for  $i \in \{0, \dots, \underline{S}-1\}$, one has 
\begin{equation}\label{e:moveYlikeX}
y_{\tau(i+1)} = R A^{p_i}C R^{-1}\cdot  y_{\tau(i)},
\end{equation}
 where $p_i$ is the digit of $\ell_i$ (common throughout $J_{k,v}$), see  [\cite{CaltaKraaikampSchmidtContinEntrop}, Lemma~25 (i)].   It follows that $\mathcal T_{2,n,\alpha}^i(\mathcal D) = [\ell_i, \ell'_i)\times [e_i, y_{\tau(i)}]$  with $e_i = R E_i R^{-1}\cdot  y_{-\overline{S}(v)-1}$.   Thus,  the union of the $\mathcal T_{2,n,\alpha}^i(\mathcal D)$ accounts exactly for $\Omega_{2,n,\alpha'}\setminus \Omega_{2,n,\alpha}$ if and only if $e_j = y_{\tau(j)-1}$ for all $1\le j \le \underline{S}(k,v)$.

Let $p_0$ be the initial digit of $\ell_0(\alpha)$, again this is common throughout $J_{k,v}$.     That $e_j = y_{\tau(j)-1}$ for all $1\le j \le \underline{S}(k,v)$  is implicit in [\cite{CaltaKraaikampSchmidtContinEntrop}], but the proof is not given.   We give the proof.

In our case of $n>3$, [\cite{CaltaKraaikampSchmidtContinEntrop}, Lemma~33] implies  that $\tau(1)-1 = \tau(1+i_{\underline{S}})$.  From 
[\cite{CaltaKraaikampSchmidtContinEntrop}, Lemma~34],   $\tau(\underline{S}-i_2) = \underline{S}$  and equivalently $i_{\underline{S}} = \underline{S}- i_2$.   Therefore, $\tau(1+i_{\underline{S}}) = \tau(\underline{S}- (i_2 -1))$. By definition, $y_{\tau(\underline{S}- (i_2 -1))} = - \ell_{i_2-1}(\zeta_{k,v})$.   Thus,  $y_{\tau(1) - 1} = - \ell_{i_2-1}(\zeta_{k,v})$.
Thus \eqref{e:yActionDigitsRelation} implies  $(R A^{p_0} C R^{-1})^{-1}\cdot y_{\tau(1) - 1} = -A^{p_0} C\cdot \ell_{i_2-1}(\zeta_{k,v})$.  By  [\cite{CaltaKraaikampSchmidtContinEntrop}, Corollary~44],  $\ell_0(\eta_{k,v}) = \ell_{i_2}(\zeta_{k,v})$ and  $\ell_{i_2}(\zeta_{k,v}) = A^{p_0-1} C \cdot \ell_{i_2-1}(\zeta_{k,v})$.  Thus, $(R A^{p_0} C R^{-1})^{-1}\cdot y_{\tau(1) - 1} = -A\cdot \ell_0(\eta_{k,v}) = -r_0(\eta_{k,v})$.  Therefore, $y_{\tau(1) - 1} = R A^{p_0} C R^{-1}\cdot  -r_0(\eta_{k,v})$.

Now, suppose that some $e_j = y_{\tau(j)-1}$ with $1\le j<\underline{S}(k,v)-1$.   By \eqref{e:moveYlikeX}, if both $\ell_j(\zeta_{k,v})$ and $\ell_{i_{\tau(j)-1}}(\zeta_{k,v})$ share the same digit, $p_j$, then $e_{j+1} = y_{\tau(j+1)-1}$.   If they do not share the same digit, then by the shape of $\Omega_{2,n,\alpha}$ it can only be that the digit corresponding to $\ell_{i_{\tau(j)-1}}(\zeta_{k,v})$ is rather $p_j-1$ and in fact  $\ell_{i_{\tau(j)}}(\zeta_{k,v})$ is the least entry of the initial orbit segment of $\ell_0(\zeta_{k,v})$ in the $p_j$ cylinder, while  $\ell_{i_{\tau(j)-1}}(\zeta_{k,v})$ is the greatest in the  $p_j-1$ cylinder.    The first of these is sent by $A^{p_j}C$ to $\ell_2$, the second by $A^{p_j-1}C$ to $\ell_{\underline{S}}$.    Thus,  $\tau(j+1) = 2$ and $e_{j+1} = RAR^{-1}\cdot y_{\tau(\underline{S})}$.   But,   [\cite{CaltaKraaikampSchmidtContinEntrop}, Lemma~34] gives  $y_1 = RAR^{-1}\cdot y_{\tau(\underline{S})}$.  We thus have  that $e_{j+1} = y_{\tau(j+1)-1}$ in this case as well.

 \smallskip
 \noindent
 {\it Step 2: Correct addition.}   
 Since $T_A\cdot x = x +t$, we have $\mathcal A = [r_0, r'_0)\times [RAR^{-1}\cdot r_{-\overline{S}(v)-1}, RAR^{-1}\cdot y_1]$.    By [\cite{CaltaKraaikampSchmidtContinEntrop}, Lemma~25 (iv) and Lemma~29~(v)], $\mathcal A = [r_0, r'_0)\times [y_{-1}, y_{\underline{S}(k,v)+1}]$.   Thereafter, the fact that $-r_0(\zeta_{k,v})$ is the second lowest height, thus that it equals $y_{-\overline{S}}$,  follows from [\cite{CaltaKraaikampSchmidtContinEntrop}, Lemma~45].  This allows arguments directly analogous to those in Step 1 to succeed.  
\end{proof}

\begin{Def}\label{d:theMiddle}    Suppose $n>3$.  
The {\em middle matching interval} is 
\[\mathcal M_n = \begin{cases} J_{1, (n-2)/2}&\text{if}\;\; 2|n;\\
                                                  J_{1, (\frac{n-3}{2}, 1, \frac{n-3}{2})}&\text{otherwise.}
                              \end{cases}
\]
\end{Def}

We show that the middle matching interval does contain the midpoint of the parameter interval, and that the mass is constant along this matching interval.   Recall from \S~\ref{ss:Correction} that for odd $n$, \cite{BKS} defined $R_n$ to be the positive root of $x^2- (2-t_{2,n}) x - 1$.   If $n$ is even, let $R_n = 1$. 

\medskip
\begin{Lem}\label{l:middleContainsOneHalf}    Suppose $n>3$.  
The value $\alpha = 1/2$ belongs to $\mathcal M_n$. Furthermore,    
\[\mathcal M_n =  \bigg[1- \frac{R_n}{t_{2,n}}, 1+\frac{R_n}{t_{2,n}}\bigg).\]
 Moreover, the function   $\alpha \mapsto  \mu(\Omega_{2,n,\alpha})$ is constant on $\mathcal M_n$, of value 
\[ \mu(\Omega_{2,n,\alpha}) = \begin{cases}  2 \ln [\frac{ 1 + \cos \pi/n}{  \sin  \pi/n}]&\text{if}\;\; 2|n;\\
\\
                                                                         2 \ln \bigg(\dfrac{ 1 +R_n}{2 \sin \frac{\pi}{2 n}} \bigg)&\text{otherwise}.
                                                 \end{cases}
\]
\end{Lem}
\begin{proof} Recall from that for the  symmetric  Rosen maps, $T_{2,n,1/2}$,  we have $T_{2,n,1/2}(-x) = - T_{2,n,1/2}(x)$.   As in \S~\ref{ss:symmRosen}, this allows us to promote results about the Rosen maps themselves to the $T_{2,n,1/2}$.     For ease let $v_n$ be defined such that $\mathcal M_n = J_{1, v_n}$. 

The paper \cite{BKS} shows that   $R_n$ gives the highest $y$-value of their planar domain for the Rosen map of index $n$.   See \S~\ref{ss:symmRosen}, we then have that this is the highest $y$-value  of $\Omega_{2,n,1/2}$ and also $-1$ times that value gives the lowest $y$-value.  The basic results of \cite{BKS} and symmetry show that the $T_{2,n,1/2}$-orbits of of $r_0(1/2)$ and $\ell_0(1/2)$ are in agreement with both values being sent to $x=0$  upon applications of $R_{1, v_n}$ and $L_{1, v_n}A$, respectively.   This shows that $\alpha = 1/2$ does lie in $\mathcal M_n$, and in fact is the special midpoint value (some author refer to such a value as a `pseudocenter', here it is the center).    Symmetry also shows that there is the same number of entries in each of the $T_{2,n,1/2}$-orbits of $r_0(1/2)$ and $\ell_0(1/2)$ before they meet at $x=0$. 
 
 Since $R_n$ gives the above mentioned highest values,   we have that  $\ell_0(\zeta_{1, v_n}) = -R_n$.    Since for any $\alpha$ one has $\alpha = 1 + \ell_0(\alpha)/t$,  the right endpoint of each $\mathcal M_n$ is as claimed.  Symmetry gives the value of its left endpoint.

Since the $T_{2,n,1/2}$-orbits of $r_0(1/2)$ and $\ell_0(1/2)$  meet at $x=0$ after the same number of steps, we have  $\overline{S}(v_n) = \underline{S}(1, v_n)$.
By Lemma~\ref{l:massIncreasesWhere}, $\mu(\Omega_{2,n,\alpha})$ is hence constant along $\mathcal M_n$ for each $n$.  
As mentioned in \S~\ref{ss:symmRosen}, $\mu(\Omega_{2,n,1/2})$ equals twice the $\mu$-mass of the planar domain for the Rosen map of index $n$. With the correction of \S~\ref{ss:Correction}, we find the claimed expressions for $\mu(\Omega_{2,n,\alpha})$.
\end{proof}

\begin{Lem}\label{l:massChangeExactlyWhere}    Suppose $n>3$.  
The function   $\alpha \mapsto  \mu(\Omega_{2,n,\alpha})$ decreases on any matching interval $J_{k,v}$ such that $J_{k,v}$ lies to the left of $\mathcal M_n$, is constant on $\mathcal M_n$, and increases on any other matching interval. 
\end{Lem}
\begin{proof} We already have the constancy of mass function $\alpha \mapsto  \mu(\Omega_{2,n,\alpha})$ along the middle matching interval, for each $n$. 

For any $n$,   the symmetry $T_{2,n,1-\alpha}(-x) = - T_{2,n,\alpha}(x)$ allows one to find that on each matching interval to the right of $\mathcal M_n$ there is a corresponding  matching interval to the left of $\mathcal M_n$, but with a reversal of steps in left- and right-orbits (up to matching).  In particular, if the mass function $\alpha \mapsto  \mu(\Omega_{2,n,\alpha})$ decreases on matching intervals to the left of $\mathcal M_n$ then it increases on those to the right of $\mathcal M_n$.   That is, we need only show the claimed result for matching intervals to the left of $\mathcal M_n$.

For any $n$ and any $v$, from \eqref{e:rKv} we find that  $\overline{S}(v) = \sum_{i=1}^s\, c_i + \sum_{i=1}^{s-1}\, d_i$.   Each appearance $W$ of \eqref{eq:defW}  in $L_{k,v}$ contributes $n-2$ to $\underline{S}(k,v)$.   Thus, each appearance of 
$ \tilde{\mathcal C} = \tilde{\mathcal C}_k = W^{k-1} (A^{-1}C)^{-1}$ contributes $-1 + (k-1)(n-2)$, and each appearance of 
$ \tilde{\mathcal D} = \tilde{\mathcal C}_{k+1}$ contributes $-1 + k(n-2)$.  From \eqref{eq:defL}, 
\[ \underline{S}(k,v) = (n-3) +  [ (k-1) n-2 k+ 1)](-1+  \sum_{i=1}^{s}\, c_i) +  [ kn-2 k  -1] \sum_{i=1}^{s-1} \,d_i + (k-1)(n-2)\,.\]

With $n>3$ and $v$ fixed, this is clearly an increasing function of $k$.  With $k=2$, we have 
$\underline{S}(2,v) =  (n- 3)(\sum_{i=1}^{s}\, c_i) +  [ 2 n-5] \sum_{i=1}^{s-1} \,d_i + (n-2)$.  Since $n>3$,  this is greater than $\overline{S}(v)$.  Thus for all $k\ge 2$ we have that $\underline{S}(k,v)- \overline{S}(v) >0$.   By Lemma~\ref{c:massIncreasesWhere}, we have that the mass function $\alpha \mapsto \mu(\Omega_{2,n,\alpha})$ decreases along any $J_{k,v}$ with $k\ge 2$.

We turn to the setting of $k=1$.  Note that 

\begin{equation}\label{e:diffS}
\underline{S}(1,v) - \overline{S}(v)= (n-3)( 1 +  \sum_{i=1}^{s-1} \,d_i) +  1-  2\,\sum_{i=1}^{s}\, c_i\,.
\end{equation}
 If $s=1$ this becomes $\underline{S}(1,c_1)  - \overline{S}(v)= n-2 -  2 c_1$.
(Recall that  $c_1 <n-2$ when $k=1$.)   Thus,   $\underline{S}(1,c_1)-\overline{S}(c_1) \ge 0$ if and only if $(n-2)/2 \ge c_1$.  Equality holds exactly when $c_1 = (n-2)/2$ and of course this is possible exactly when $n$ is even.

We can now easily show that the case of $n=4$ holds.  Indeed,  here $\mathcal M_4 = J_{1,1}$, which is a leftmost subinterval of $I_{1,1}$.  Hence, any matching interval to the left of $\mathcal M_4$ is contained in the union of the  $I_{1,k}, k\ge 2$.   Therefore,  we do have that the mass function decreases along all matching intervals to the left of $\mathcal M_4$.   

For each $n$, we have that $\underline{S}(1,v) - \overline{S}(v)$  is zero when  $v$ is the defining word of $\mathcal M_n$.   
As well, for fixed $v$, \eqref{e:diffS} shows that $\underline{S}(1,v) - \overline{S}(v)$ strictly increases with $n$.   Thus for $n=5$ we must merely show that $\underline{S}(1,v) - \overline{S}(v)$ is positive between $J_{1,1}$ and $\mathcal M_5 = J_{1,111}$.  But, as \cite{CaltaKraaikampSchmidt} shows, the matching intervals between these two are the $J_{1,1q1}$ with $q \ge 2$ and the $J_{1,v}$ with $v$ running through the descendants of these various $1q1$.   Now, with $n>4$ we have that $\underline{S}(1,1) - \overline{S}(1)$ is positive, and to pass from $v=1$ to $v= 1q1$, we add $q$ to the sum of the $d_j$ and $2$ to that of the $c_j$.  Again from \eqref{e:diffS},  we find $( \underline{S}(1,1q1) - \overline{S}(1q1)) - (\underline{S}(1,1) - \overline{S}(1) ) =  (n-3)q - 4$. Since $n=5$, with $q\ge 2$  this is positive.   In any $v \in \mathcal V$ with $c_1=1$ all of the $c_j=1$.   A consequence of the construction of the tree of words $\mathcal V$, see Definition~\ref{d:thetaQ}, is that upon fixing such a $q$, any descendant of $1q1$  has its $d_j$ all in $\{q, q+1\}$. Since any descendant is formed by appending a string of alternating $d_j$ and $c_{j+1}$,   we see that the difference similar to the above is again positive.   That is, the result also holds for $n=5$.   

Now suppose that the result holds for some odd $n = 2h+3$ with $h\ge 1$ and consider $n' = n+1 = 2(h+2)$.    We aim to  show that, with $n'$ as our indexing value,  $\underline{S}(1,v) - \overline{S}(v) >0$ for all matching intervals between $J_{1,h1h}$  and $J_{1, h+1}$, knowing that $\underline{S}(1,h1h) - \overline{S}(h1h)$ is positive.     These intermediate matching intervals are all of the form $J_{1,v}$ with $v$ a descendant of $h1h$.   

If $h=1$ then as above each such descendant $v$ has all $c_j =1$ and all $d_j \in \{1,2\}$  we pass from $111$ to $v$ by appending an equal number of $c_j$ and $d_j$ values.  The positivity of $\underline{S}(1,v) - \overline{S}(v)$ thus holds.   If $h\ge 2$ then each  descendant has all $d_j =1$ and all $c_j \in \{h,h+1\}$.   Thus, the change in the difference of the right hand side of \eqref{e:diffS} as we add one pair of $d_j, c_{j+1}$ at a time is $(n'-3)d_j - 2 c_{j+1} \ge n'-3 - 2(h+1) = 2(h+2)-3 - 2(h+1) = 0$.  We conclude that the difference $\underline{S}(1,v) - \overline{S}(v)$ is positive for all of these $v$.

We now aim to show the result for $n'' = n+2 = 2(h+1)+3$.   Hence, we must show, with our value $n''$, the positivity along the matching intervals between  $J_{1, h+1}$ and $J_{1,(h+1)1(h+1)}$, knowing that \eqref{e:diffS}  is positive when $v= h+1$.  That is, we must show positivity when $v$ is any descendant of $h+1$ that is not a descendant of $(h+1)1(h+1)$.  The children of $h+1$ are of the form $(h+1)(1 h)^q1 (h+1)$.   Thus for each $q>0$ we add $(q+1)h +1$ in total to the sum of the $c_j$-values, and  $q+1$ to the sum of the $d_j$.    Thus \eqref{e:diffS} changes by $(n''-3)(q+1) - 2 (q+1)h -2=   q-1$.
Any descendant of $v = (h+1)(1 h)^q1 (h+1)$ with $q>0$  can be formed by repeatedly appending pairs of the form $1h$ or $1(h+1)$ to $v$.   Each such step results in the difference changing by   $(n''-3) - 2h= 2$ or by $1$.  Thus, the result holds for $n''$.   

Thus, an induction proof shows the result holds for all $n>3$.
\end{proof}

\subsection{Continuity of $\alpha \mapsto \mu(\Omega_{2,n,\alpha})$}\label{ss:conMuOm}   The table [\cite{CaltaKraaikampSchmidtContinEntrop}, Table~1] also points to  definitions  of  $\Omega_{3,n, \alpha}$ in the cases that $\alpha$ is either endpoint of a matching interval; in short, when compared to $\Omega_{3,n, \alpha}$ with $\alpha$ in the interior of the matching interval,  at the left endpoint there is a `missing' leftmost lower rectangle, and at the right a missing rightmost upper rectangle, see [\cite{CaltaKraaikampSchmidtContinEntrop}, Figures~7, 8].  These definitions are easily adjusted to the case of $m=2$, and  the proofs given in  \cite{CaltaKraaikampSchmidtContinEntrop} succeed. 

The remaining values of $\alpha$ are non-matching values.   The arguments of  [\cite{CaltaKraaikampSchmidtContinEntrop}, \S~6] can be directly adapted to the setting of $m=2$.  There, one has 
that $\Omega_{3,n,\alpha}$  includes a central rectangle  see [\cite{CaltaKraaikampSchmidtContinEntrop}, Figure~9], given by the union of the full cylinders $\Delta_{3,n, \alpha}(k')$   times  a constant $y$-fiber  $\Phi$ equal to the closure of $\mathbb I_{3,n,\alpha}$.  The full $\Omega_{3,n,\alpha}$ is the closure of the union of the $\mathcal T_{3,n,\alpha}$-orbit of a certain $\mathcal Z$ formed by the union of the  central rectangle with the left remaining subinterval times a constant $y$-fiber $\Phi^{-}$   equal to negative one times the closure of the right remaining subinterval  and similarly for a remaining rectangle of $y$-fiber denoted $\Phi^{+}$.      Each non-matching value is the limit of endpoints of matching intervals; for each such endpoint $\alpha'$, there is an analogous $\mathcal Z_{\alpha'}$.   The various $\mathcal Z_{\alpha'}$ nicely converge to $\mathcal Z_{\alpha}$ and one shows that the union of the images under the $\alpha'$-map does indeed converge to the union of the $\mathcal T_{3,n,\alpha}$-orbit of  $\mathcal Z_{\alpha}$.    All of this, in adapted form,  carries over   to show the continuity of $\alpha \mapsto \mu(\Omega_{2,n,\alpha})$ when $0< \alpha <1$.
  
\begin{Cor}\label{c:massIncreasesWhere}    Fix $n\ge 3$.   The function   $\alpha \mapsto  \mu(\Omega_{2,n,\alpha})$:  decreases on $0< \alpha < 1- \frac{R_n}{t_{2,n}}$; is constant on $\big[1- \frac{R_n}{t_{2,n}}, 1+\frac{R_n}{t_{2,n}}\big)$;  and, increases on  $1+\frac{R_n}{t_{2,n}}\le \alpha <1$.    Its maximal value is as given in Lemma~\ref{l:middleContainsOneHalf}. 
\end{Cor}
\begin{proof} Since the mass function is continuous on $(0,1)$, the result follows from Lemmas~\ref{l:middleContainsOneHalf} and \ref{l:massChangeExactlyWhere}. 
\end{proof}

Since the maximal value of $\mu(\Omega_{2,n,\alpha})$  goes to zero as $n \to \infty$, we have completed the proof of Theorem~\ref{t:mIsTwoSimpleEntropyBehavior}.

\subsection{Planar extensions   for $\alpha = 0, 1$}\label{ss:omZeroOne}   We also give planar extensions for the endpoints of the parameter intervals.

\subsubsection{Planar extensions $\Omega_{2,n, \alpha}$ for  $\alpha = 0$}\label{sss:PlanarForZero}   For each $n\ge 3$, the domain $\Omega_{3,n, 0}$ is given in 
 \cite{CaltaSchmidt}.    We now sketch the  adjustment to the setting of $m=2$.   
 
 Fix $m=2$ and  $n \ge 3$,  and set $t = t_n = t_{2,n} = 2 \cos(\pi/n)$.  Recycle notation and let 
$g(x) = k t - 1/x$, with $k$ the unique integer such that $g(x) \in [-t, 0)$.   Now,  $-t \le k t - 1/x < 0$ implies $-(k+1) t \le -1/x < -k t$  (and since $x<0$ we have $-1/x >0$ so $k <0$).     That is,  $\frac{1}{(k+1) t} \ge x > \frac{1}{k t}$.  We find that the cylinders corresponding to $-k,  k\ge 2$ are all full, whereas (if $n>3$) the cylinder of $k = -1$ is sent to $[- t +1/t, 0)$.      One finds that the orbit of $x=-t$ is periodic of length $n-2$ and corresponds to following the admissible word  $W = A^{-2} C (A^{-1} C)^{n - 3}$.   The reader may be reassured by the facts that both  [\cite{CaltaKraaikampSchmidt}, equation (2.1)] holds when $m=2$, and (projectively) $W = \begin{pmatrix}1+ t^2&t^3\\-t& 1-t^2\end{pmatrix}$  as in  [\cite{CaltaSchmidt},  equation (3.1)]. 

It is easily verified that $\mathcal T_{2, n=3, 0}$ is bijective up to null sets on $\Omega_{2, 3, 0} := [-1,0]\times [0,1]$.  We thus assume $n>3$ in what follows.    Let $\ell_i = g^i(-t)$ for $0\le i \le n-3$. We have $-t = \ell_0< \ell_1 < \dots <\ell_{n-3}$ and $\{\ell_0, \dots, \ell_{n-4}\} \subset \Delta(-1)$ while $\ell_{n-3} \in \Delta(-2)$, where we use simplified notation for cylinders.   Since $A^{-2} C \cdot \ell_{n-3} = -t$, in fact $\ell_{n-3} = C\cdot t = -1/t$. 

Now let $y_1 = 1/t$, a value fixed by $RWR^{-1}$.  Further, let $y_i = RA^{-1}\cdot y_{i-1}$ for $1 \le i \le n-2$.    We find that $RA^{-2}\cdot y_{n-2} = y_1$.  Thus, $y_{n-2} = t$.    Let 

\[\Omega = \Omega_{2, n, 0}  =  ([\ell_{n-3}, 0]\times [0,t]) \cup \bigcup_{i=0}^{n-4} ([\ell_i, \ell_{i+1}) \times [0, y_{i+1}])\,,\]
see Figure~\ref{f:Omega2FiveZero}  for the case of $n=5$.

We have that  $\mathcal S': (x,y) \mapsto (-y,-x)$ sends the pair $(\ell_{n-3}, y_{n-2})$ to $(\ell_0, y_1)$.    It is straightforward to check that for any invertible $M = \begin{pmatrix} a, &b\\-b &d\end{pmatrix}$ that  $\mathcal T_{M}( \mathcal S'(x,y)) = \mathcal S' \circ  \mathcal T_{M^{-1}}(x,y)$ for any $(x,y)$.    Since   $(\ell_1, y_2) = \mathcal T_{A^{-1}C}(\ell_0, y_1)$ and $(\ell_{n-4},  y_{n-3})  = \mathcal T_{(A^{-1}R)^{-1}} (\ell_{n-3}, y_{n-2})$, by recursion we find that the vertices of $\Omega$ are preserved by $\mathcal S'$.   That is,  $\Omega$ is preserved by $\mathcal S'$.   Notice that when $n$ is odd we have  $(\ell_{(n-3)/2}, y_{(n-1)/2}) = (-1, 1)$.    Since there are vertices lying along the curve $y = -1/x$,  one sees that $\mu(\Omega)$ is infinite. 

\subsubsection{Planar extensions for accelerated $T_{2,n, 0}$}\label{sss:PlanarForAcc} 
Similar to \cite{CaltaSchmidt}, we now accelerate $g=T_{2,n, 0}$ by inducing past the cylinder of $W$.  We treat only the case of $n>3$ and hence $t>1$, the remaining case is readily handled.  Let $\epsilon_0 = W^{-1}\cdot 0 = t^3/(1-t^2)$ and define $f: [-t,0) \to [-t,0)$ by $f(x) = g(x)$ when $x>\epsilon_0$ and for $-t<x\le \epsilon_0$ let $f(x) = W^{j(x)}\cdot x$ where $j(x)$ is minimal such that $W^{j(x)}\cdot x>\epsilon_0\,$.    

Let $\mathcal F$ be the planar map associated,  in our usual manner,  to $f$.
We give a bijectivity domain for  $\mathcal F$.   For this, let $\epsilon_i$ with $1 \le i \le n-3$ be the initial $g$-orbit of $\epsilon_0$, that is, $\epsilon_i = (A^{-1}C)^i\cdot \epsilon_0$.  Since $W \cdot \epsilon_0 = 0$, we have that $\epsilon_{n-3}$ is the right endpoint of the cylinder $\Delta(-2)$.   Also,  $\ell_i < \epsilon_i<\ell_{i+1}$ for $0 \le i \le n-4$ and $\ell_{n-3}<\epsilon_{n-3}$; recall that $\ell_{n-3}$ is sent to $-t$ and hence it is the left endpoint of $\Delta(-2)$.    Let $z_1 = RWR^{-1}\cdot 0$.    Note that $\mathcal T_{W^{k+1}}$ sends the rectangle $(W^{-k-1}\cdot \epsilon_0, W^{-k}\cdot \epsilon_0)\times [0, z_1]$ to lie directly above the $\mathcal T_{W^{k}}$-image of 
$(W^{-k}\cdot \epsilon_0, W^{-k-1}\cdot \epsilon_0)$.    Since $y_1$ is a fixed point of $RWR^{-1}$ one finds the union over  all $k$ of these images is the rectangle $[\epsilon_0, 0]\times [z_1, y_1]$.  With this as a guide, we now let 
\[\Omega = \big([\ell_0, \epsilon_0)\times [0, z_1]\big) \cup \big([\epsilon_{n-3}, 0) \times [0, t] \big)  \cup \bigcup_{i=0}^{n-4}  \big([\epsilon_i, \epsilon_{i+1}) \times [0, y_i] \big),  \]
see Figure~\ref{f:Omega2FiveZero}  for the case of $n=5$. 

One finds that $\Omega = \Omega_{2,n,0} \setminus \big(\mathcal D_1 \cup \mathcal D_2\big)$ where 
\begin{equation}\label{e:deletionRects} \mathcal D_1 = \bigcup_{i=0}^{n-3}\,\mathcal T_{2,n,0}^{i}\big([\ell_0, \epsilon_0)\times [z_1, y_1]\big)\;\text{and}\; \mathcal D_2 = \bigcup_{i=1}^{n-3}\,\mathcal T_{2,n,0}^{i}\big([\ell_0, \epsilon_0]\times [0, z_1]\big). 
\end{equation}
Given that $\mathcal F$ maps $[\ell_0, \epsilon_0)\times [0, z_1]$ bijectively to $[\epsilon_0, 0]\times [z_1, y_1]$ and the bijectivity of $\mathcal T_{2,n,0}$ on $\Omega_{2,n,0}$,  the bijectivity of $\mathcal F$  on $\Omega$ will follow as soon as we show that $RA^{-2}CR^{-1}$ maps $y_{n-3}$ to $z_1$.    Given the matrix form of $W$, we find $z_1 = R W \cdot \infty = -1/((1+t^2)/-t) = t/(1+t^2)$.  On the other hand,  $y_{n-3} = (RA^{-1}C R^{-1})^{-1}\cdot t$ and thus $RA^{-2}CR^{-1}\cdot y_{n-3} = RA^{-1} R^{-1}\cdot t = -1/((-t-1/t))$.   The equality holds and hence $\Omega$ is a domain of bijectivity for $\mathcal F$.

\def\bigG{1.618}%
\def\litG{0.618034}
\def\lGSqrd{0.381966}
\def\epZero{-1.171}
\def\epOne{-0.764}
\def\epTwo{-0.309}
\def\zOne{0.447}
\def\zTwo{0.854}
\def\zThree{1.31}
\begin{figure}[h]
\scalebox{.5}{
\begin{tikzpicture}[scale=8]
\draw[fill=black!20] (-\bigG,0)--(-\bigG,\litG)--(-1,\litG)--(-1,1)--(-\litG,1)--(-\litG, \bigG)--(0, \bigG)--(-0,0) -- cycle;
\draw[pattern=north west lines]  (-\bigG,\zOne)--(-\bigG,\litG)--(\epZero,\litG)--(\epZero,\zOne) -- cycle;
\draw[pattern=north west lines]  (-1,\zTwo)--(-1,1)--(\epOne,1)--(\epOne,\zTwo) -- cycle;
\draw[pattern=north east lines]  (-1,\litG)--(-1,\zTwo)--(\epOne,\zTwo)--(\epOne,\litG) -- cycle;
\draw[pattern=north west lines]  (-\litG,\zThree)--(-\litG,\bigG)--(\epTwo,\bigG)--(\epTwo,\zThree) -- cycle;
\draw[pattern=north east lines]  (-\litG,\zThree)--(-\litG,1)--(\epTwo,1)--(\epTwo,\zThree) -- cycle;
\node at (-0.45,0.6) {$\mathcal B_{-2}$};
\draw[dashed] (-\litG, 0)--(-\litG,1); 
\draw[dashed] (\epTwo, 0)--(\epTwo, \bigG); 
\node at (-0.22,0.6) {$\mathcal B_{-3}$};
\draw[dashed] (-0.146, 0)--(-0.146, \bigG); 
\node at (-0.06,0.6) {\large{$\dots$}};
\node[below] at (-\bigG, 0) {$(\ell_0,0) = (-t, 0)$};
\node[below] at (\epZero, 0) {$(\epsilon_0,0)$};
\node[below] at (\epTwo, 0) {$(\epsilon_2,0)$};
\node[left] at (-\bigG, \zOne) {$(\ell_0,z_1)$};
\node[left] at (-\bigG, \litG) {$(\ell_0,y_1)$};
\node[left] at (-\litG,\zThree) {$(\ell_2, z_3)$};
\node[above] at (-\litG,\bigG) {$(\ell_2, y_3)$};
\node[above, left] at (-1,\zTwo) {$(\ell_1, z_2)$};
\node[above, left] at (-1,1) {$(\ell_1, y_2)$};
\node[above] at (\epOne, 1) {$(\epsilon_1, y_2)$};
\node[above] at (0, \bigG) {$(0, y_3)= (0, t)$};
\node[below] at (0, 0) {$(0, 0)$};
\foreach \x/\y in  {-\bigG/0, -\bigG/\zOne,-\bigG/\litG, -1/\zTwo,-1/1,  -\litG/\zThree,-\litG/\bigG, 0/\bigG,0/0, \epZero/0, \epOne/1,\epTwo/0}  
 \node at (\x,\y) {$\bullet$}; 
\end{tikzpicture}
}
\caption{The domain $\Omega_{2,5,0}$ is the union of the gray and the hatched regions.   These latter are excised to leave the planar domain of the two-dimensional system defined by the interval map given by  by accelerating  $T_{2,5,0}$ by applying appropriate powers of $W$ for $x$ in $(\ell_0, \epsilon_0) = (-t_{2,5}, \epsilon_0)$, where $\epsilon_0 = W^{-1}\cdot 0$. The negative sloped hatched rectangles form $\mathcal D_1$ of \eqref{e:deletionRects}, the  positive sloped hatched rectangles form $\mathcal D_2$.}
\label{f:Omega2FiveZero}
\end{figure}
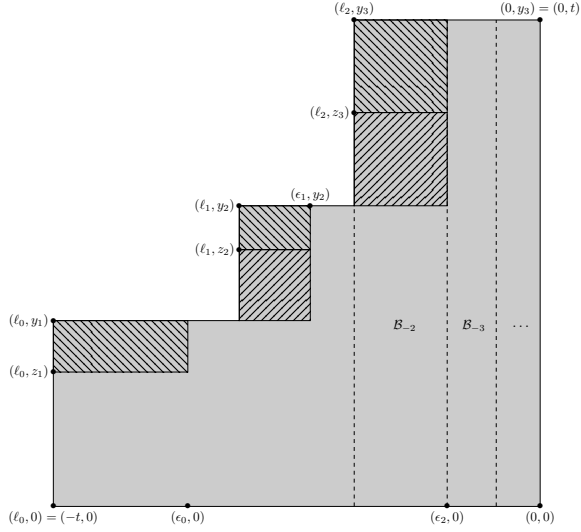
 
\subsubsection{Planar extension for $(m,n,\alpha) = (2,n,1)$}    Applying the map $\mathcal S$ to $\Omega_{2, n, 0}$ results in $\Omega_{2, n, 1}$ and thus one can verify the following.

In  [\cite{CaltaKraaikampSchmidt}, Section~3]  one finds a discussion of $\Omega_{3,n, 1}$ for $n\ge 3$. 
For our setting of $m=2$, we must check some of that argument,  since here $\mu = 0$, but $1/\mu$ appears in that earlier discussion.  Recycling notation, we have $\mathbb I = [0,t)$ and  the non-full cylinder $\Delta(1) = [1/t, t)$ whose image under $f$ is $A C\cdot  [1/t,t) = t + [-t, -1/t) = [0, t-1/t) = [0, (t^2-1)/t)$.    The other cylinders are full; on the cylinder index by $k>1$,  $T$ takes the form $x \mapsto k t -1/x$.

The argument of [\cite{CaltaKraaikampSchmidt}, Section~3] goes through to show that $A^2 C (A C)^{n-3} = A(A C)^{n-2} = A (C A^{-1})^2 = A C A^{-1} C A^{-1}$ fixes $t$ and the orbit of $t$ is given by $t \mapsto A C \cdot t \mapsto (A C)^2\cdot t \mapsto\cdots \mapsto (A C)^{n-3}\cdot t$, followed by  $A^2 C\cdot (\, (A C)^{n-3}\cdot t\,) = t$.    We thus let $s_i = (A C)^i\cdot t, 0\le i \le n-3$ and set   

\begin{equation}\label{e:om2N0}
\Omega  = \bigcup_{i=1}^{n-2}   \,( [s_i, s_{i-1}]\times [-1/s_{i-1}, 0]) = ([0, 1/t)\times [-t,0])\cup \bigcup_{i=1}^{n-3}  \,( [s_i, s_{i-1}]\times [-1/s_{i-1}, 0])\,.
\end{equation}

For $k>1$, by [\cite{CaltaKraaikampSchmidtPfsErgodicity}, Proposition~12], the two dimensional $\mathcal T$ maps the  rectangle  $\Delta(k)\times [-t, 0]$ below the image of $\Delta(k+1)\times [-t, 0]$ so that they share a horizontal line. Since $R\, A^k C\, R\cdot y = R A^k \cdot y = -1/( k t + y)$ we find that the union of these images, with $2 \le k < \infty$ is $\mathbb I \times  [-1/t, 0)$.   One easily verifies that $\mathcal T_{A C}$ sends $[s_i, s_{i-1}]\times [-1/s_{i-1}, 0]$ to  $[s_{i+1}, s_i]\times [-1/s_i, -1/t]$ for $i<n-2$.    Hence,   $\Omega$ is a bijectivity domain for $\mathcal T$ up to null sets.    We thus hereafter denote it by $\Omega_{2,n,1}$.

Note that $\Omega_{2,3,1}$ is the square $[0,1] \times [-1,0]$.   
 
\newpage 

\subsection{Continuity of entropy, constancy of `Rohlin integrals'}\label{ss:conEntConRoh} 

\subsubsection{Agreement of first return maps}   Using `word processing' we show  that for each $0 \le \alpha<1$ that upon  an the appropriate subset of negative $x \in \mathbb I_{\alpha}$  first returns under $T_{\alpha}$ are given by  compositional powers of $T_0$.   Compare the following with [\cite{CaltaKraaikampSchmidtPfsErgodicity}, Lemma~142]. 

\begin{Lem}\label{t:wordPro}  Fix $m=2$ and $n\ge 3$.   For each $\alpha \in [0,1)$ and suppose $x \in \mathbb I_{\alpha}$ is negative and that there exists $m\ge 0$ such that both $T_{\alpha}^{m}(x)<0$ and $T_{\alpha}^{m+1}(x)<0$.   Then there is some $k\in \mathbb N$ such that $T_{0}^{k}(x) =  T_{\alpha}^{m+1}(x)$.    Furthermore,  for any $y$ one has $\mathcal T_{0}^{k}(x,y) =  \mathcal T_{\alpha}^{m+1}(x,y)$. 
 \end{Lem} 
\begin{proof}   With minor adjustments, the proof of [\cite{CaltaKraaikampSchmidtPfsErgodicity}, Lemma~142] succeeds in our setting, as we now show.

 We can and do assume that $m$ is minimal with respect to the hypotheses.    From the definition of our maps, we can write
\[ T_{\alpha}^{m+1}(x) = A^{a_{m+1}}C A^{a_{m}}C \cdots A^{a_1}C \cdot x,\]
with both $a_{m+1}, a_1$ negative.   If $m=0$, then $A^{a_1}C^{c_1}\cdot x = T_0(x)$ and $k=1$.   We now assume that $m>0$ and  
 perform word processing on $A^{a_{m+1}}C A^{a_{m}}C \cdots A^{a_1}C$ with the goal to achieve an expression for this element of the group $G_n$ that is in an admissible form for $T_0$.  
 
 We process from right to left,  using the following substitution rules.  We start with the augmented word  `$A^{a_{m+1}}C A^{a_{m}}C \cdots A^{a_1}C$;'  and   use the substitutions:  

\[
 \begin{aligned} 
(i)&\;\;\;\;\;\;\;\;\;   A^{a_1} C;                       &\mapsto &\;\;\;\;\;  A;A^{a_1-1}C \\
(ii)&\;\;\;\;\;\;\;\;\;   A^a C A; \;               &\mapsto &\;\;\;\;\; A C A; W^{a-1} \\  
(iii)&\;\;\;\;\;\;\;\;\;  A^a C  A C A; \;                   &\mapsto &\;\;\;\;\; \begin{cases}  A C A; W^{a-2}A^{-2}C(A^{-1}C)^{n-4}&\text{if}\;a\ge 2, n\ge 4;\\
                                                                                                                        A C A; W^{a-2}A^{-1}&\text{if}\;a\ge 2, n=3;\\
                                                                                                                        ;A^{-2}C(A^{-1}C)^{n-4}&\text{if}\;a = 1. 
                                                                                                                        \end{cases}&\\  
(iv)&\;\;\;\;\;\;\;\;\;   \; A^a C;  A^b C\; &\mapsto &\;\;\;\;\; A C A; W^{a-1} A^{b-1}C \;\phantom{please align,th}\text{if} \;a >0\\    
(v)&\;\;\;\;\;\;\;\;\;     \;  A^{a_{m+1}} C;\;     &\mapsto &\;\;\;\;\; ;A^{a_{m+1}} C \\    
(vi)&\;\;\;\;\;\;\;\;\;     \;  A^{a_{m+1}} C A C A ;\;     &\mapsto &\;\;\;\;\; ;A^{a_{m+1}-1} C (A^{-1}C)^{n-3}\,. \\     
  \end{aligned}
  \]
  Each rule, up to ignoring the  {\em marker} ``;", is the result of applying a group identity. Mainly, it is a version of $A^kCA = CA W^k$ from  Lemma~\ref{l:shortRightId}.    After a rule is applied, we multiply on the left by the next unprocessed $A^{a_i}C$,  and then determine which rule is then to be applied.  (In the `middle case' of $(iii)$, we also first multiply the $A^{-1}$ with the term to its right.) The processing halts upon applying (v) or (vi).   One easily verifies that processing does halt, confer Figure~\ref{f:wordProAuto}.  The result is an admissible $T_0$-word, which acts so as to send $x$ to $T_{\alpha}^{m+1}\cdot x$. 
 
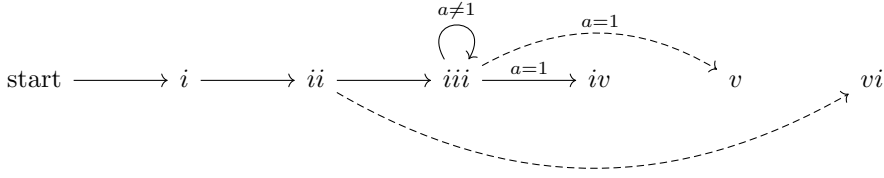
\begin{figure}
\begin{tikzcd}[column sep=3pc,row sep=2pc]
\text{start}  \ar{r}& i \ar{r} &ii \ar{r}\ar[dashed, bend right]{rrrr}&iii \ar{r}{a=1}\ar[dashed, bend left]{rr}{a=1}\arrow[loop, out=120, in=60, distance = 2em]{}{a\neq1}&iv &v&vi\\
\end{tikzcd}
\caption{Sequence of word processing relations to pass from $T_{\alpha}$-admissible word to $T_0$-admissible word,  used in proof of  Lemma~\ref{t:wordPro}. 
Processing halts when the demarcation symbol is to the left of the word.}
\label{f:wordProAuto}
\end{figure}
\end{proof}

 We follow the next step in the arguments of \cite{CaltaKraaikampSchmidtPfsErgodicity},   and seek a region for which the {\em first} return under $\mathcal T_0$ agrees with that of $\mathcal T_{\alpha}$.  To this end, we restrict the region to which we return, so that $\{-1,-2\}$ digits introduced in the word processing could never index  an `early return' of a $\mathcal T_0$-orbit. 
 
\begin{Lem}\label{l:firstReturnTo2ndQuadrant}  Suppose that $0<  \alpha< 1$.  Let $d = \min\{-4, -2 + d_{\alpha}(\ell_0(\alpha))\}$,  $N_d = \bigcup_{i \le d}\, \Delta_{\alpha}(i)$  and  $\Omega^{+}_{\alpha, d} =   \Omega_{\alpha} \cap \{(x,y) \mid y>0, x<0,  x \in N_d\; \text{and}\; x' \in N_d, {\text where}\; \mathcal T_{\alpha}^{-1}(x,y) = (x',y')\,\}$.   
  The first return maps of  $\mathcal T_0$ and of $\mathcal T_{\alpha}$ to $\Omega^{+}_{\alpha, d}$ agree. 
\end{Lem}

\begin{proof}  Since $\Omega^{+}_{\alpha, d}$ lies in the second quadrant of the $(x,y)$-plane, one finds   $\Omega^{+}_{\alpha, d} \subset \Omega_0$ and hence both return maps are defined.   Suppose that  $(x,y) \in \Omega^{+}_{\alpha, d}$ and  $\mathcal T_{\alpha}^{m+1}(x,y)$ is its first $\mathcal T_{\alpha}$-return. 
From the previous lemma,   this first return   is also  {\em some} return under $\mathcal T_0$.  

 If $m=0$ then certainly $T_{\alpha}(x) = T_0(x)$, and the first return of the two-dimensional maps  also agree.  
 Now assume that $m>0$.  
Suppose  
 \[ T_{\alpha}^{m+1}(x) = A^{a_{m+1}}C^{c_{m+1}}A^{a_{m}}C^{c_{m}} \cdots A^{a_1}C^{c_1}\cdot x.\]   
Of course,   if all  $a_i$ are negative, then the orbits are the same and the result clearly holds.  In general, there may be intermediate consecutive appearances of negative exponents.  We thus add:
 \[
(vii)\;\;\;\;\;  A^{a'} C  A^a C;\;     \mapsto     A^{a'} C;  A^a C \;\;\text{if}\;   a<0, a'<0,                    
 \]
 and process so as to achieve an admissible $T_0$ expression.

 Whenever any of our substitution rules inserts $T_0$-steps,  these insertions are always of corresponding $0$-digits  in $\{-1, -2\}$.   Of course,  any $x$-value  with either of these as its $0$-digit cannot have  $\alpha$-digit less than or equal to $d$.   Thus, these insertions cannot cause an early $\mathcal T_0$ return. 
 
 Our rules also make changes in the exponents of appearances of $A$.       However,  any such exponent is changed at most once throughout the process, and if it is changed then it is decreased exactly by one.   Since each substitution realizes a group identity,  it follows that  any element in $N_d$ that is in the $T_0$-orbit segment of $x$  but not already present in the $T_{\alpha}$-orbit segment is either (1) isolated (with respect to this property),  or (2) is sent by $T_0$ to  $T_{\alpha}^{m}(x)$ while its $T_0$-orbit predecessor is not in $N_d$.    But, by hypothesis, $\mathcal T_{\alpha}^{m}(x, y) \notin \Omega^{+}_{\alpha, d}$.   Therefore the result holds.  \end{proof} 
 
 \subsubsection{ Entropy times mass is constant, proof of Theorem~\ref{t:mIsTwoSimpleEntropyBehavior}}
\begin{Prop}\label{p:continuityOfEntropyOnIntervals}    Fix $n \ge 3$.   The function   $\alpha \mapsto  h(T_{2, n ,\alpha}) \mu(\Omega_{2, n, \alpha})$ is constant on   $(0, 1)$.  
\end{Prop}
\begin{proof}  Fix $\alpha, \alpha' \in (0,1)$ and let $\Omega_{\alpha, \alpha'} = \Omega^{+}_{\alpha, d} \cap \Omega^{+}_{\alpha', d'}$, with the natural interpretation of the notation of    Lemma~\ref{l:firstReturnTo2ndQuadrant}.  That lemma implies that the first returns to this region  of $\mathcal T_{2,n,\alpha}$ and $\mathcal T_{2,n,\alpha}$  agree.  
 These hence define the same dynamical system, and Abramov's formula yields that their entropy can be expressed as 
\[ \dfrac{h(\mathcal T_{2,n,\alpha}) \, \mu( \Omega_{2,n,\alpha})}{\mu(\Omega_{\alpha, \alpha'})} =   \dfrac{h(\mathcal T_{2,n,\alpha'}) \, \mu( \Omega_{2,n,\alpha'})}{\mu(\Omega_{\alpha, \alpha'})}.\]
Since the entropy of a dynamical system equals the entropy of the natural extension system, the result holds.   
\end{proof} 
 
Theorem~\ref{t:mIsTwoSimpleEntropyBehavior} now follows from  Proposition~\ref{p:continuityOfEntropyOnIntervals} and Corollary~\ref {c:massIncreasesWhere}.

 \subsubsection{ Rohlin integrals are constant}  

We continue to follow \cite{CaltaKraaikampSchmidtContinEntrop}.   Recall the cocycle $\tau$ from Definition~\ref{def:Tau}.  
\begin{Def}\label{d:tau}  
  Fix $n \ge 3$.   For $ \alpha \in [0,1]$ and for each $x \in \mathbb I_{2,n, \alpha}$,  let $\tau_{\alpha}(x) = \tau_{2,n, \alpha}(x) = - 2 \log \vert c x + d\vert$ where $T_{2,n, \alpha}(x) = (a x + b)/(c x + d)$.   Of course, this is simply  $\tau_{\alpha}(x) =  \tau(M, x)$ where $T_{2,n, \alpha}(x) = M\cdot x$. 
  
   For any $n, \alpha$, we call the integral $\int_{\Omega_{2,n,\alpha}} \tau_{2,n,\alpha}(x)\,  d \mu$ a {\em Rohlin integral}.   See the first displayed equation in the following proof for a justification of this term. 
\end{Def}

  We now show, in terms of simplified notation,  that the integral of $\tau_{\alpha}(x)$ over $\Omega_{\alpha}$ with respect to $d \mu$ gives the product of  $h(T_{\alpha})$ with  $\mu(\Omega_\alpha)$.   Furthermore,  these integrals are constant with respect to $\alpha$, including at the endpoints, this although both $\mu(\Omega_0)$ and $\mu(\Omega_1)$ are infinite.
 
\begin{Lem}\label{l:integralsMatch}  Fix $n \ge 3$.   For $0 <\alpha < 1$,  
 \[  h(T_{2,n,\alpha})\, \mu(\Omega_{2,n,\alpha})  = \int_{\Omega_{2,n,\alpha}} \tau_{2,n,\alpha}(x)\,  d \mu = \int_{\Omega_{2,n,0}} \tau_{2,n,0}(x)\,  d \mu =\int_{\Omega_{2,n,1}} \tau_{2,n,1}(x)\,  d \mu. 
\]
\end{Lem}

\begin{proof}     For $0 < \alpha < 1$,   Rohlin's formula  gives 
\[  h(T_{2,n,\alpha}) = \int_{\mathbb I_{2,n,\alpha}}\, \log \vert T'_{2,n,\alpha}(x) \vert\, d \nu_{2,n,\alpha}=  \int_{\mathbb I_{2,n,\alpha}}\, \tau_{2,n,\alpha}(x) d \nu_{\alpha} =  \dfrac{\int_{\Omega_{2,n,\alpha}}\, \tau_{2,n,\alpha}(x) d \mu}{\mu(\Omega_{2,n,\alpha})},
\] 
where the last equality holds because $\nu_{2,n,\alpha}$ is the marginal measure for the probability measure on $\Omega_{2,n,\alpha}$ induced by $\mu$.   Since  $\mu(\Omega_{2,n,\alpha})< \infty$, our first equality holds.  

By Theorem~\ref{t:mIsTwoSymm}  and the discussion of symmetry in \S~\ref{ss:Symmetry}, it suffices to prove only the first of the two remaining equalities. 
   
 We employ a tower construction, with base $\Omega^{+}_{\alpha, d}$.  Throughout, we use the convention that equalities of sets are all considered up to $\mu$-null sets.    Poincar\'e recurrence applied to each of $\mathcal T_{2,n,\alpha}^{\pm 1}$ on this positive measure subspace of $\Omega_{2,n,\alpha}$ shows the bijectivity of the first return of $\mathcal T_{2,n, \alpha}$ to $\Omega^{+}_{\alpha, d}$.    For each $k \in \mathbb N$ we define $R_k \subset \Omega^{+}_{\alpha, d}$ as the set of points whose first return to $\Omega^{+}_{\alpha, d}$  is given by applying $\mathcal T_{2,n,\alpha}^k$.   The $R_k$ thus partition $\Omega^{+}_{\alpha, d}$ and also $\Omega^{+}_{\alpha, d} = \cup_{k \in \mathbb N}\, \mathcal T_{2,n, \alpha}^{k}(R_k)$ up to $\mu$-null sets.      Now set $S = \cup_{k \in \mathbb N}\, \sqcup_{i=0}^{k-1} \mathcal T_{2,n, \alpha}^{i}(R_k)$.   We have that  $\mathcal T_{2,n, \alpha}$ bijectively maps the tower $S$ to itself.   By the ergodicity of $\mathcal T_{2,n,\alpha}$, we deduce that  $S = \Omega_{2,n,\alpha}$. 

Let  $\rho$ denote projection onto the $x$-coordinate of the first return map on $\Omega^{+}_{\alpha, d}$.   
 Due to  the cocycle property of $\tau(M, x)$ and the additivity of integration, one deduces that 
\[ \int_{\Omega_{2,n,\alpha}} \tau_{2,n,\alpha}(x)\,  d \mu = \int_{\Omega^{+}_{\alpha, d}}  \log \vert \rho'(x)\vert\,  d \mu.\]

Since $\Omega_{2,n,0}$ has infinite measure, we cannot simply repeat this argument to directly achieve the equality with $\Omega_{2,n,0}$ replacing $\Omega_{2,n,\alpha}$.    Rather, we turn to the analog of the induced system that \cite{CaltaSchmidt} associates to the $\mathcal T_{3,n,0}$.    Associated to the accelerated system given in \S~\ref{sss:PlanarForZero} is the induced system on $\mathcal I = \Omega_{2,n,0}\setminus \cup_{i=0}^{n-3}\, \mathcal T_{2, n, 0}^{i}(\mathcal D)$, where  $\mathcal D$ is the domain fibered over $[-t, \epsilon)$.   Since 
$\Omega_{2,n,0}$ meets the curve $y = -1/x$ exactly on the initial $n-2$ points of the $ \mathcal T_{2, n, 0}$-orbit of $(-t, 1/t)$,  this induced system is a finite $\mu$-measure dynamical system.   Now,   $\mathcal D$ is contained in the $\mathcal T_{2, n, 0}$-image of the blocks of digit $-3$ and less. That is, the   forward $\mathcal T_0$ orbits of the points of $\mathcal I$ up to first return to $\mathcal I$ gives all of $\Omega_{2,n,0}$. (As usual, these statements are true up to null sets, for example the purely periodic orbit of $(-t, 1/t)$ is not contained in those forward orbits.)   But, again by the cocycle property of $\tau(M,x)$ and the additivity of integration,   the integration of the  logarithm of the absolute value of the first derivative of the accelerated one-dimensional map over $\mathcal I$ equals  $\int_{\Omega_{0}} \tau_{2,n,0}(x)\,  d \mu$.    

Now, for our $0<\alpha<1$, due to the restriction on the digits in the definition of  the region $\Omega^{+}_{\alpha, d}$, the region is contained in $\mathcal I$.  Hence,  Lemma~\ref{l:firstReturnTo2ndQuadrant} implies that the first return maps of these two finite area systems agree.  Again the cocycle property gives that the Rohlin integrals of these systems are equal.  By transitivity of equality, we have  
$\int_{\Omega_{2,n,\alpha}} \tau_{2,n,\alpha}(x)\,  d \mu = \int_{\Omega_{2,n,0}} \tau_{2,n,0}(x)\,  d \mu$.      
\end{proof} 
 
 \bigskip
 Lemma~\ref{l:integralsMatch} combines with \eqref{e:symmRosenEntr} to achieve the proof of Theorem~\ref{t:mIsTwoGotVol}.

\section{From $\Omega_{2,n,1}$ to $\Omega_{3,n,1}$}\label{s:the Shift}
\subsection{The domain of bijectivity  $\Omega_{3,n,1}$}  
In [\cite{CaltaKraaikampSchmidtPfsErgodicity}, Proposition~12] one finds that for each $n \ge 3$ the region which in our notation is $\Omega_{3,n,1}$ is 
\begin{equation}\label{e:Om3n1} \Omega_{3,n,1} = ([0,1]\times [-1,0] )\;  \cup \;  \bigcup_{i=1}^{n-2}\; [r_i, r_{i-1}]\times [-1/r_{i-1},0],
\end{equation} 
where $r_0, r_1, \dots, r_{n-2}$ is the $T_{3,n,1}$-orbit of $r_0 = t = t_{3,n}$.   See [\cite{CaltaKraaikampSchmidtPfsErgodicity}, Figure~3].   

Directly above their proposition, they state what we can express, with $C_3$ dependent on $m=3$,  and $A_3$ denoting the translation by $t_{3,n}$, 
as $r_i = (A_3 C_{3}^{2})^{n-2}\cdot t$ for $1 \le i \le n-2$, with $r_{n-2}\cdot t = 1$; and,  $T^{n-1}(t)= A_3 C_{3} (A_3 C_{3}^{2})^{n-2}\cdot t = t$.

\subsection{Sending $\Omega_{2,n,1}$ to the right hand part of $\Omega_{3,n,1}$}\label{ss:theShift}  Fix $n \ge 3$, and note that $t_{3,n} = 1 + t_{2,n}$.

Fix $M_1 = \begin{pmatrix}1&1\\0&1\end{pmatrix}$.  We now consider the $\mathcal T_M$ image of  $\Omega_{2,n,1}$.    This function acts on $x$-coordinates as   translation by $1$.      We have $RM_1 R \cdot y = -1/(1 -1/y)$.  
Thus,  $\mathcal T_{M_1}(t_{2,n}, -1/t_{2,n}) = (t_{3,n}, -1/t_{3,n})$.   

Now, $C_3 = \begin{pmatrix}1&-1\\1&0\end{pmatrix}$, and thus $C_{3}^{2}M_1 = C_2$,
 where of course we are letting $C_2 =  \begin{pmatrix}0&-1\\1&0\end{pmatrix}$.  Since  $A_3 = M_1 A_2$ (where we let $A_2$ denote that translation by $t_{2,n}$),  we have that $M_{1}^{-1} A_3 C_{3}^{2}M_1 = A_2 C_2$.   Thus,  $\mathcal T_{M_1}$ conjugates $\mathcal T_{A_2 C_2}$ to $\mathcal T_{A_3 C_{3}^{2}}$.    Referring back to \eqref{e:om2N0} and its discussion, we find that $\mathcal T_{M_1}$ sends $\Omega_{2,n,1}$ to $\Omega_{3,n,1}\setminus ([0,1]\times [-1,0])$.  In other words,  $\Omega_{3,n,1} = ([0,1]\times [-1,0]) \cup \mathcal T_{M_1}(\Omega_{2,n,1})$, see Figure~\ref{f:mIs2ToRightPartMis3}.

\begin{figure}[h]
\scalebox{.45}{
\begin{tikzpicture}[scale=5]
\filldraw[blue!40] (0,0)--(\bigG,0)--(\bigG, -\litG)--(1, -\litG)--(1, -1)--(\litG, -1)--(\litG, -\bigG)--(0,-\bigG)--cycle;
\foreach \x/\y in {0/0, \bigG/0, \bigG/-\litG, 1/-1, \litG/-\bigG, 0/-\bigG%
} { \node at (\x,\y) {$\bullet$}; } 
\node[left] at (0, 0) {$(0,0)$};
\node[above, left] at (0, -\bigG) {$(0,-G)$};
\node[above] at (\bigG, 0) {$(G,0)$};
\node[below] at (\bigG, -\litG) {$(G,-g)$};
\node[below] at (1,-1) {$(1,-1)$};
\filldraw[shift={(3,0)}, gray!40] (0,0)--(1,0)--(1,-1)--(0,-1)--cycle;  
\filldraw[shift={(3,0)}, blue!40] (1,0)--(1+\bigG,0)--(1+\bigG, -\litG*\litG)--(2, -\litG*\litG)--(2, -0.5)--(\bigG, -0.5)--(\bigG, -\litG)--(1, -\litG)--cycle; 
\foreach \x/\y in {3/0, 4+\bigG/0, 4+\bigG/-\litG*\litG, 5/-0.5, 3+\bigG/-\litG, 4/-1, 3/-1%
} { \node at (\x,\y) {$\bullet$}; } 
\node[left] at (3, 0) {$(0,0)$};
\node[left] at (3, -1) {$(0,-1)$};
\node[below] at (4, -1) {$(1,-1)$};
\node[above] at (4+\bigG, 0) {$(1+G,0)$};
\node[below] at (4+\bigG, -\litG*\litG) {$(1+G,-g^2)$};
\node[below] at (5, -0.5) {$(2,-1/2)$};
\node[below] at (3+\bigG, -\litG) {$(G,-g)$};
\end{tikzpicture}
}
\caption{For each $m\ge 3$, the map $\mathcal T_{M_1}$, with $M_1 = \begin{pmatrix}1&1\\0&1\end{pmatrix}$, sends  the planar domain $\Omega_{2,n,1}$ to $\Omega_{3,n,1}\setminus \{(x,y) \,|\,  x\ge 1\}$, see \S~\ref{ss:theShift}.  Here,  $n$=5.   On the left: The planar domain $\Omega_{2,5,1}$     On the right: $\Omega_{3,5,1}$ with the $\mathcal T_{M_1}$-image of $\Omega_{2,5,1}$ highlighted.}
\label{f:mIs2ToRightPartMis3}
\end{figure}
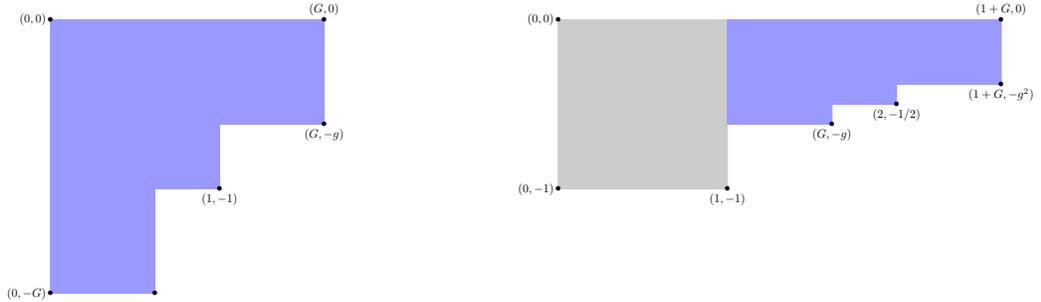

 \subsection{Change of variable calculation}    Due to  the invariance of $\mu$ under  $\mathcal T_M$ in general and thus under $\mathcal T_{M_1}$ in particular,    for each $n \ge 3$  we have
 \[
\int_{\Omega_{3,n,1}\setminus ([0,1]\times [-1,0])} \tau_{3,n,1}(x)\,  d \mu = \int_{\Omega_{2,n,1}} \tau_{3,n,1}(M_1\cdot x)\,  d \mu.
\]
 
 Now, for $x \ge  1$ one has $\tau_{3,n,1}(x) = \tau(\, A_{3}^k C_{3}^{2},  x)$ for some $k \ge 1$.    Since   both $\tau(A_3, x)$ and $\tau(A_2, x)$ vanish for any $x$,    the cocycle property of $\tau(M, x)$ gives $\tau_{3,n,1}(M_1\cdot x) = \tau(A_{3}^k C_{3}^{2} M_1, x) =   \tau(C_{3}^{2} M_1, x) = \tau(C_{2} , x) = \tau(A_{2}^{\ell}C_{2}, x)$ for  any $\ell\in \mathbb Z$ and any $0\le x < t_{2,n}$.   
Therefore,  
\[\int_{\Omega_{3,n,1}\setminus ([0,1]\times [-1,0])} \tau_{3,n,1}(x)\,  d \mu = \int_{\Omega_{2,n,1}} \tau_{2,n,1}(x)\,  d \mu.
\]

 Finally, Theorem~\ref{t:mIsTwoGotVol}  gives 
 \begin{equation}\label{e:compOfSquareLittleVol} 
 \int_{\Omega_{3,n,1}\setminus ([0,1]\times [-1,0])} \tau_{3,n,1}(x)\,  d \mu = \text{vol}(T^1(G_{2,n}\backslash \mathbb H)).
\end{equation}

  \section{Proof of the conjecture}\label{s:theConjectureHolds!}  
 
 \subsection{Conjecture reduction given in  \cite{CaltaKraaikampSchmidtContinEntrop}}
 By [\cite{CaltaKraaikampSchmidtContinEntrop}, Theorem~4], for each $n \ge 3$,  the function  $\alpha \mapsto  h(T_{3,n,\alpha}) \mu(\Omega_{3,n,\alpha})$ is constant on $(0,1)$.  Furthermore, for all such $\alpha$, by [\cite{CaltaKraaikampSchmidtContinEntrop}, Lemma~149],
  \[  h(T_{3,n,\alpha}) \mu(\Omega_{3,n,\alpha})  = \int_{\Omega_{3,n,\alpha}} \tau_{3,n,\alpha}(x)\,  d \mu = \int_{\Omega_{3,n,0}} \tau_{3,n,0}(x)\,  d \mu =\int_{\Omega_{3,n,1}} \tau_{3,n,1}(x)\,  d \mu. 
  \]  
Thus, to prove the first statement of Theorem~\ref{t:Main}, it suffices to show that $\int_{\Omega_{3,n,1}} \tau_{3,n,1}(x)\,  d \mu$ equals the volume of the unit tangent bundle of $G_{3,n}\backslash \mathbb H$.       
  
 As mentioned above,  [\cite{CaltaKraaikampSchmidtPfsErgodicity}, Proposition~12] shows that $\Omega_{3,n,1}$ has the form of \eqref{e:Om3n1} and in particular contains the square  $[0,1]\times [-1,0]$.    By [\cite{CaltaKraaikampSchmidtContinEntrop}, Lemma~151], 
 \begin{equation}\label{e:littleVol}   \int_{[0,1]\times [-1,0]}  -2 \log  x     \; d\mu = \pi^2/3. 
\end{equation} 
As pointed out in the proof of [\cite{CaltaKraaikampSchmidtPfsErgodicity}, Proposition~12], 
for $x<1$ one has $T_{3,n,1}(x) = A_{3}^k C_3\cdot x$, with $k \in \mathbb Z$.  Hence, for such $x$,   $\tau_{3,n,1}(x) = -2 \log  x $.  It follows that 
\[\int_{\Omega_{3,n,1}} \tau_{3,n,1}(x)\,  d \mu = \pi^2/3 + \int_{\Omega_{3,n,1}\setminus ([0,1]\times [-1,0])} \tau_{3,n,1}(x)\,  d \mu.\]

 \subsection{Final argument}
From this last,    \eqref{e:compOfSquareLittleVol}  gives 
\begin{equation}\label{e:rohlinIntIsvolSum} \int_{\Omega_{3,n,1}} \tau_{3,n,1}(x)\,  d \mu = \pi^2/3 + \text{vol}(T^1(G_{2,n}\backslash \mathbb H)).
\end{equation}
 
  From \eqref{e:volForm}, trivial algebra  shows that 
$\text{vol}(T^1(G_{3,n}\backslash \mathbb H)) = \pi^2/3 + \text{vol}(T^1(G_{2,n}\backslash \mathbb H))$. Combined with \eqref{e:rohlinIntIsvolSum}, this shows 
\[ \int_{\Omega_{3,n,1}} \tau_{3,n,1}(x)\,  d \mu = \text{vol}(T^1(G_{3,n}\backslash \mathbb H))\]
and the conjecture holds!

\end{document}